\documentclass[12pt]{amsart}

\usepackage{amssymb,amscd,amsthm, verbatim,amsmath,color,fancyhdr, mathrsfs}
\usepackage{graphicx}
\usepackage{turnstile}
\usepackage{multicol}
\usepackage{enumerate}
\usepackage{tikz-cd}
\usepackage{dsfont}
\usepackage{bbm}

\usepackage[english]{babel}
\usepackage[T1]{fontenc}
\usepackage{hyperref}

\usepackage{xcolor}

\usepackage[letterpaper, left=2.5cm, right=2.5cm, top=2.5cm,
bottom=2.5cm,dvips]{geometry}
\newcommand{\bb}[1]{\mathbb{#1}}
\newcommand{\spc}[0]{\text{ }}

\newcommand{\N}[0]{\mathbb{N}}
\newcommand{\Z}[0]{\mathbb{Z}}

\newcommand{\R}[0]{\mathbb{R}}
\newcommand{\E}[2]{\mathbb{E}_{#1 \in [#2]}}

\newcommand{\norm}[1]{\left\Vert #1\right\Vert}

\newcommand{\abs}[1]{\mathopen{}\left| #1\mathclose{}\right|}

\newcommand{\floor}[1]{{\left \lfloor #1 \right \rfloor}}

\newcommand{\T}[0]{\mathbb{T}}
\newcommand{\C}[0]{\mathbb{C}}

\newtheorem{thm}{Theorem}[section]
\newtheorem{prop}[thm]{Proposition}

\newtheorem{lem}[thm]{Lemma}

\newtheorem{rem}[thm]{Remark}
\newtheorem{defi}[thm]{Definition}

\theoremstyle{definition}
\newtheorem*{defi*}{Definition}
\makeatletter
\g@addto@macro{\endabstract}{\@setabstract}
\newcommand{\authorfootnotes}{\renewcommand\thefootnote{\@fnsymbol\c@footnote}}%
\makeatother
\begin{document}

\begin{center}
  \LARGE 
  Further bounds in the polynomial Szemer\'edi theorem \par \bigskip

  \normalsize
  \authorfootnotes
  Bartłomiej Bychawski\footnote{bartlomiej.bychawski@student.uj.edu.pl, bsb36@cam.ac.uk}\textsuperscript{1}\textsuperscript{2},
  Bartosz Głowacki\footnote{bartosz16.glowacki@student.uj.edu.pl}\textsuperscript{1}
  and
  Ivan Spyrydonov\footnote{ivan.spyrydonov@doctoral.uj.edu.pl}\textsuperscript{1}
  \par \bigskip

  \textsuperscript{1}Faculty of Mathematics and Computer Science, Jagiellonian University \par
  \textsuperscript{2}Faculty of Mathematics, University of Cambridge\par 
  \bigskip

  \today
\end{center}

\begin{abstract}
We show that there exists $c > 0$ such that any subset of $\{1,\ldots,N\}$ having size $\gg N / \exp( (\log\log\log N)^c )$ contains a nontrivial pattern of the form $x,x+y,x+2y,x+y^3$. It is the first configuration of complexity strictly greater than $0$, other than refinements of arithmetic progressions, for which quantitative bounds over integers were obtained.
\end{abstract}

\tableofcontents

\section{Introduction}

Given a sequence of polynomials $P_1,\ldots,P_m \in \Z[y]$, we denote the size of the largest subset of $\{1,\ldots,N\}$ which does not contain a nontrivial pattern
$
x,x+P_1(y),\ldots,x+P_m(y)
$ 
by $r_{P_1,\ldots,P_m}(N)$. Applying Furstenberg principle and results from ergodic theory Bergelson and Leibman showed \cite{PolWdW} that whenever $P_1(0) = \ldots = P_m(0) = 0$, we obtain the bound
$$
r_{P_1,\ldots,P_m}(N) = o_{P_1,\ldots,P_m}(N).
$$
This result can be viewed as a generalization of the famous Szemer\'edi theorem \cite{Sze}, hence it has been refereed to as \textit{Polynomial Szemer\'edi Theorem}.

Their result is purely qualitative and fails to provide any information about the rate of convergence. The problem of obtaining quantitative bounds was first formulated by Gowers \cite[Problem 11.4]{Gowers2001Arithmetic} and it has received a lot of attention recently \cite{Pre17, PP19, Pel20, PSS23, AS25, SW25}, following the initial breakthrough of Peluse \cite{Pel19b}. In this paper we prove the  following  quantitative bound.

\begin{thm}\label{main result}
      Let $N \in \mathbb{N}$, $A \subseteq \{1,2,\ldots,N\}$, and assume that $A$ does not contain any nontrivial configurations of the form 
        $
        x, x+y, x+2y, x+y^3
        $.
        Then there exists  $c > 0$ such that
        $$
\abs{A} \ll \frac{N}
                {
                    \exp
                    (    
                        {
                            (
                            \log  \log  \log N
                            )
                        }^c
                    )    
}.
        $$
\end{thm}

An argument analogous to the one presented in this paper would show that a similar bound holds if of one replaces $y^3$ by any polynomial $P \in \Z[y]$ satisfying $P(0)=0$ and $\deg P \geq 3$.

Previous quantitative bounds regarding polynomial Szemer\'edi theorem concerned configurations of complexity $0$ (cf. \cite[Definition 1.9]{Kuca2021True}) \cite{PP19, Pel20, SW25} or arithmetic progressions with specified polynomial difference \cite{Pre17, PSS23, AS25}. 

The later case was tackled in the full generality using a transference method which differs largely from the degree-lowering methods developed by Peluse \cite{Pel19b}. Those methods may ever be applicable only to so called \textit{transferable patterns} \cite[Definition 1.4]{AS25} such as $x,x+y^2,x+2y^2$ or $x,x+y,x+y^2,x+y+y^2$. Obtaining quantitative estimates with transference is still a highly nontrivial task and as for now this method was successfully applied only to arithmetic progressions with arbitrary polynomial differences. 

One of the easiest examples of a nontransferable pattern is $x,x+y,x+2y,x+y^2$, which is a simultaneous consideration of the three term arithmetic progression and the nonlinear Roth configuration $x,x+y,x+y^2$. This configuration however remains out of reach, as each of its terms has complexity at least $1$, hence the best inverse theorems for the counting operator
$$
\Lambda_{y,2y,y^2}^{N,M}(f_0,f_1,f_2,f_3) = \sum_{x \in [N]} \sum_{y \in [M]} f_0(x) f_1(x+y) f_2(x+2y) f_3(x+y^2),
$$
where $M \asymp N^{1/2}$, would involve the $u^3$-norms of $f_0,f_1,f_2$ and the $U^2$-norm of $f_3$ \cite[Theorem 1.6]{Kuca2021True}. The main difficulty here is that the transference is not applicable, and original methods which involved generalized von Neumann inequality and degree-lowering were previously always arriving at $U^1$-inverse theorems and deducing some version of the density increment lemma from them. Since in the above configuration one can only expect weaker inverse theorems, this method is inapplicable. 

This higher complexity issue described above is a consequence of the quadratic relation
$$
( x^2 + 2x) + (x+2y)^2 = 2(x+y)^2 + 2(x+y^2).
$$
Changing the pattern to $x,x+y,x+2y,x+y^3$ prevents such relations from forming, yet it preserves some features of the previous configuration. In particular, both original and new patterns look like an arithmetic progression of small scale and a point very far away from it. This suggests, that the above pattern carries similar scale difficulties as the previously described nontransferable pattern, yet has a simpler algebraic structure. For that reason, we treat this modified pattern as a testing ground for methods one would need to apply to solve the original one.

Our argument is based on the work of Kuca \cite{Ku20}. He obtains quantitative bounds for configuration $x,x+y,x+2y,x+y^3$ over the simpler model of finite fields. He obtains the following discorelation estimate of the counting operator
\begin{multline*}
\underset{x,y \in \bb{F}_p}{\bb{E}} f_0(x) f_1(x+y) f_2(x+2y) f_3(x+y^3) = 
\underset{x,y,z \in \bb{F}_p}{\bb{E}} f_0(x) f_1(x+y) f_2(x+2y)
f_3(z) + O(p^{-c}),
\end{multline*}
where $p$ is a sufficiently large prime, and $f_i \colon \bb{F}_p \rightarrow \C$ are 1-bounded.
That estimate suggests that one can obtain $U^2$-inverse theorems for the first three functions $f_0,f_1,f_2$ and the desired $U^1$-inverse theorem for the last $f_3$ term. Indeed, in the process of proving Theorem \ref{main result} we derive the following inverse theorems.

\begin{thm}[Inverse theorem for $f_0,f_1$ and $f_2$] \label{WSTEP: inverse thm dla pierwszych trzech funkcji}
If $N \in \N$ and $\delta > 0$ satisfy ${\delta}^{-O(1)} \ll N$, then for all 1-bounded functions $f_0, f_1, f_2, f_3 \ \colon \ \bb{Z} \rightarrow \bb{C}$ supported on $[N]$, the bound
	$$
	\Bigg{|} \sum_{x \in \Z} \sum_{y \in \N} \spc
	f_0({x}) f_1({x} + y) f_2({x} + 2y) f_3({x} + y^3)
	\Bigg{|} \geq \delta N^{4/3},
	$$
    implies the existence of positive integers $q \ll \delta^{-O(1)}$ and $N' \gg \delta^{O(1)} N^{1/3}$ such that for each $i \in \{0,1,2\},$ 
    $$
    \sum_x \norm{f_i}^{4}_{U^2(x+q\cdot[N'])} \gg \delta^{O(1)} N N'^3.
    $$
\end{thm}

\begin{thm}[Inverse theorem for $f_3$] \label{WSTEP: inverse thm dla ostatniej funkcji} If $N \in \N$ and $\delta > 0$ satisfy ${\delta}^{-O(1)} \ll N$, then for all 1-bounded functions $f_0, f_1, f_2, f_3 \ \colon \ \bb{Z} \rightarrow \bb{C}$ supported on $[N]$ the bound
	$$
	\Bigg{|} \sum_{x \in \Z} \sum_{y \in \N} \spc
	f_0({x}) f_1({x} + y) f_2({x} + 2y) f_3({x} + y^3)
	\Bigg{|} \geq \delta N^{4/3},
	$$
	implies the existence of positive integers  $q \ll \delta^{-O(1)}$ and $N' \gg \delta^{O(1)} N^{1/3}$ such that
    $$
    \sum_{x} 
    \Bigg|
    \sum_{y \in [N']} f_3(x+qy)
    \Bigg|
    \gg
    \delta^{O(1)}NN'.
    $$
\end{thm}

As expected, Theorem \ref{WSTEP: inverse thm dla pierwszych trzech funkcji} is a $U^2$-inverse theorem and Theorem \ref{WSTEP: inverse thm dla ostatniej funkcji} is the $U^1$-inverse theorem. However, since both linear and cubic scales are involved, we obtain a largeness of the average of $U^2$-norms or $U^1$-seminorms over many arithmetic progressions having small common difference and length $N' \approx N^{1/3}$.

\subsection{Notation} 
\noindent For real $R\geq 1$, write $[R]=\{1,2,\ldots, \lfloor R \rfloor \}$ for a one-sides interval starting from $1$ and $[\pm R]=\{-\lfloor R \rfloor, \ldots, -1,0,1, \ldots , \lfloor R \rfloor\}$ for two sided interval centered at $0$.

A complex valued function $f \colon X\rightarrow \C$ is 1-bounded if for every $x \in X$ we have $\abs{f(x)} \leq 1$. \\
A function $f\colon \Z \rightarrow \C$ is called $C-$Lipschitz along arithmetic progressions of common difference $q$ if for every $x,y \in \Z$ we have
$$ \abs{f(x)-f(x+qy)}\leq C\abs{y}.$$
We use the notation 
$$\sum_{x_1, \ldots, x_n} a_{x_1, \ldots, x_n} :=\sum_{x_1, \ldots, x_n \in \Z} a_{x_1, \ldots, x_n}.
$$

We denote $\T :=\R / \Z$ and for a measurable $F: \T \rightarrow \C$ and $p\in [1,\infty)$, we define
$$
\norm{F}_{L^p}:=\left( \int_{\T}\abs{F(\alpha)}^p d\alpha\right)^{1/p} = \left( \int_0^1 \abs{F(\alpha)}^p d\alpha \right)^{1/p}.
$$
We also define 
$
\norm{F}_{L^{\infty}}
:=
\mathrm{supp}
\big\{
x \in \R
:
\lambda^1
\left(
    F^{-1}
        [x,+\infty)
\right) > 0
\big\}
$ where $\lambda^1$ denotes the standard one dimensional Lebesgue measure. Moreover, for $\alpha \in \R$ we write 
$
\norm{\alpha}:=\min_{n\in \Z}\abs{n-\alpha}
$
.

For a finite set $S$ and a function $f\colon S\rightarrow \C$, we denote the average
 $$
 \underset{s\in S}{\bb{E}} \ f(s):=\frac{1}{\abs{ S}} \sum_{s\in S} f(s). 
 $$
 
 We define the Fourier transform of $f\colon \Z \rightarrow \C$ by
 $$ \widehat{f}(\alpha):=\sum_x f(x)e(\alpha x),$$
 where $e(\alpha)=e^{2\pi i \alpha}$ for $\alpha \in \T$.

 For $h\in \Z$ and $f\colon \Z \rightarrow \C$, we define the difference function $\Delta_h f \colon \Z \rightarrow \C $ as 
 $$ 
 \Delta_h f(x)
 :=
 f(x)\overline{f(x+h)}.
 $$
 Iterating, we set
 $ \Delta_{h_1, h_2, \ldots , h_s}f:=\Delta_{h_1} \ldots \Delta_{h_s}f$.
 Also, for $(h_1, h_1') \in \Z^2$ and $f\colon \Z \rightarrow \C$, we define another difference function $\Delta'_{(h_1, h_1')}f\colon \Z \rightarrow \C$ by
 $$ \Delta'_{(h_1, h_1')}f(x):=f(x+h_1)\overline{f(x+h_1')}. $$
 Iterating, we set
 $
 \Delta_{h_1, \ldots , h_s}f:= \Delta_{h_1}\ldots \Delta_{h_s} f $ 
 and
 $ 
 \Delta'_{(h_1, h_1'), \ldots (h_s , h_s')}f:=\Delta'_{(h_1, h_1')} \ldots \Delta'_{(h_s, h_s')}f
 $.
 
 Using those notations, we define Gowers $U^s$-norm as
 $$
 \norm{f}_{U^s}:=\left( \sum_{x, h_1, \ldots , h_s } \Delta_{h_1, \ldots , h_s}f(x) \right)^{1/2^s}
 $$
 and we define the box norm of f along $E_1, \ldots E_s$ to be
 $$
 \norm{f}_{\Box(E_1, \ldots , E_s)}
 :=
 {
 \left(
    \underset{{h_1, h_1' \in E_1}}{\bb{E}}
    \ldots 
        \underset{{h_s, h_s' \in E_s}}{\bb{E}}
 \sum_{x} 
        \Delta'_{
            (h_1, h_1') \ldots (h_s, h_s')
        } 
    f(x)
 \right)
 }^{1/2^s}.
 $$
 
 When $S\subset \Z$, we define the localized Gowers $U^s$-norm 
 $$ \norm{f}_{U^s(S)}:=\norm{f 1_S}_{U^s}$$
If G is a countable abelian group and $E\subset G$ is a finite and nonempty set, we define the Fej\'er kernel
$$
\mu_{E}(h):= \underset{h_1, h_2 \in E}{\mathbb{E}} \ 1_{ \{ h=h_1-h_2 \}}.
$$
When $E=[\pm H],$ we use the shorthand notation
$$
\mu_H(h)
:=
\mu_{[\pm H]}(h)
=
\underset{{h_1, h_2 \in [\pm H]}}{\bb{E}} \ 1_{ \{h=h_1-h_2\} } 
=
\frac{1}{2H+1}
{
    \left(
    1-\frac{\abs{h}}{2H+1}
    \right)
}_+. 
$$
For simplicity of notation also denote 
$
\mu_{E_1, \ldots , E_s}(h_1, \ldots , h_s):= \mu_{E_1}(h_1) \cdots \mu_{E_s}(h_s)
$ 
and similarly
$\mu_{E}(h_1, \ldots, h_s ):=\mu_{E, \ldots , E}(h_1, \ldots, h_s )
$.
Using the Fej\'er kernel notation, we can equivalently express the box norm as
$$
\norm{f}_{\Box(E_1, \ldots E_s)} 
=
{\left(
    \sum_{x, h_1, \ldots, h_s} 
    \mu_{E_1, \ldots , E_s}(h_1, \ldots , h_s)
    \Delta_{h_1, \ldots, h_s} f(x)  
\right)}^{1/2^s}
$$
after shifting $x\mapsto x-h_1'-\ldots -h_s'$ and removing the extraneous averages over $h_1', \ldots h_s'$. When a set $E$ repeats $s$ times, we often write $E^s$ instead of listing it $s$ times.

\subsection{Outline of the argument}

Now we discuss our approach in detail. We leave out the  technicalities and present statements of intermediate results in a simplified manner to make the argument easier to digest.

\subsubsection{Density increment}

We iteratively apply the following density increment result.

\begin{thm} \label{WSTEP: density increment}
    Let $N \in \mathbb{N}$, $A \subseteq [N]$, and assume that $A$ does not contain any nontrivial configurations of the form 
        $
        x, x+y, x+2y, x+y^3
        $.
  Then either 
  $$N\ll \exp\left( (\log(1/\alpha))^{O(1)} \right),$$
  or there exist positive integers $q \ll \exp\left( (\log(1/\alpha))^{O(1)} \right), \ N' \gg \exp\left( -(\log(1/\alpha))^{O(1)} \right)N^{1/3}$, and $a\in \Z$ such that
  $$ \frac{\abs{A\cap (a+q\cdot[N']) }}{N'} \gg \alpha + \Omega\big(\exp\left( -(\log (1/\alpha) )^{O(1)} \right)\big).
  $$
\end{thm}

Since $A$ does not contain any configuration of interest, our newly obtained subset does not contain $x,x+y,x+2y,x+q^2y^3$. This configuration is not exactly the same, however it is similar enough to iterate the process and keep zooming in to intersections of our original set with smaller and smaller arithmetic progressions. 

Each time we zoom in, the density of our set grows, and since it cannot exceed $1$, this process of zooming in must terminate in $O\big(\exp\left( (\log(1/\alpha))^{O(1)} \right)\big)$ many steps. Notice, however, that the only reason the process may fail is that the length of the progression to which we zoomed in is very small. Comparing sizes of this progressions after a bounded number of steps with the upper bounds yields the bound presented in Theorem \ref{main result}.

\subsubsection{Inverse theorem for the dual of the last term} The first major intermediate result we obtain is the $U^1$-control of the dual function
$$
F(x) := \underset{y \in [N^{1/3}]}{\bb{E}} f_0(x - y^3) f_1(x+y-y^3) f_2(x+2y-y^3).
$$
Employing the framework developed by Peluse \cite{Pel19b}, we obtain the following.

\begin{thm} \label{WSTEP: inverse theorem dla F}
    If $N \in \N$, and $\delta > 0$ satisfy the bound ${\delta}^{-O(1)} \ll N$, then for all 1-bounded functions $f_0, f_1, f_2, f_3 \ \colon \  \bb{Z} \rightarrow \bb{C}$ supported on $[N]$, the bound
	$$
	\Bigg{|} \sum_{x \in \Z} \sum_{y \in \N} \spc
	f_0({x}) f_1({x} + y) f_2({x} + 2y) f_3({x} + y^3)
	\Bigg{|} \geq \delta N^{4/3},
	$$
    implies the existence of a positive integer $q \ll \delta^{-O(1)}$ and a 1-bounded function $\phi\colon \Z \rightarrow \C$ which is $\ll {\delta}^{-O(1)}/N$ Lipschitz along arithmetic progressions of common difference $q$ and satisfies
    $$
	\sum_{x \in \Z} F(x) \phi(x) \gg \delta^{O(1)} N.
	$$
\end{thm}

The proof of this result begins with an application of the generalized von Neumann inequality derived by  Kravitz, Kuca and Leng \cite[Theorem 1.1]{KKL24}. After obtaining control of our counting operator by a higher $U^s$-norm of $F$, we perform the degree lowering argument very similar to the one presented in \cite{PP19}.

\subsubsection{Inverse theorems for the first three functions} We start by expanding the expression obtained from Theorem \ref{WSTEP: inverse theorem dla F} to obtain inequality
$$
	\sum_{x \in \Z} \sum_{y \in \N} \spc
	f_0({x}) f_1({x} + y) f_2({x} + 2y) \phi({x} + y^3)
	 \gg \delta^{O(1)} N^{4/3}.
$$
The function $\phi$ is not only Lipschitz along arithmetic progressions of difference $q$, but also for every $u \in [q]$ there exist numbers $\alpha_u,\beta_u \in \T$ such that 
$
\phi(u+qx) = e(\alpha_ux + \beta_u)
$. Choosing relevant residues of $x$ and $y$ modulo $q$ and considering only the sum of terms having these residues, we obtain a sum which looks similar to
$$
	\sum_{x \in [N/q]} \sum_{y \in [N^{1/3}/q]} \spc
	f_0(qx) f_1(q(x + y)) f_2 ( q(x + 2y) ) e( \alpha (x + y^3)).
$$
Making the first function absorb the part of the phase dependent on $x$, we substitute \\
$g_0(x) = f_0(qx)e( \alpha x)$, $g_1(x) = f_1(qx)$, $g_2(x) = f_2(qx)$ and $g_3(x) = e(\alpha x^3)$, hence we obtain 
$$
\sum_{x \in [N/q]} \sum_{y \in [N^{1/3}/q]} \spc
	g_0(x) g_1(x + y) g_2 (x + 2y) g_3(y).
$$
Since $g_3$ is $\ll \delta^{-O(1)}/N^{1/3}$-Lipschitz, the expression above is close to
$$
\sum_{x \in [N/q]} \sum_{y \in [N^{1/3}/q]} \spc
	g_0(x) g_1(x + y) g_2 (x + 2y).
$$
At this point we apply the van der Corput inequality several times in a way similar to the original approach of Gowers \cite{Go01} and deduce Theorem \ref{WSTEP: inverse thm dla pierwszych trzech funkcji}.

\subsubsection{Inverse theorem for the last term} Employing the inverse theorem for the first three functions (Theorem \ref{WSTEP: inverse thm dla pierwszych trzech funkcji}) allows us to replace the first three functions by $\phi_0, \phi_1,\phi_2 \colon \Z \rightarrow \C$ which are made of possibly different linear phases on different intervals of length $N^{1/3}$.
We therefore obtain the largeness of the sum looking like
$$
	\sum_{x \in \Z} \sum_{y \in \N} \spc
	\phi_0({x}) \phi_1({x} + y) \phi_2({x} + 2y) f_3({x} + y^3),
$$
where for each $i \in \{0,1,2\}$ and $n \in [N^{2/3}]$ there exists $\alpha_{i,n}, \beta_{i,n} \in \T$ such that 
$$
\phi_i(x + n\lfloor N^{1/3}\rfloor) = e(\alpha_{i,n} x + \beta_{i,n})
$$
holds for every $x \in [N^{1/3}]$.

We currently have the correlation of $f_3$ with some average of linear phases, while we aim to obtain a correlation with a suitably Lipschitz function along many intervals of length $\approx N^{1/3}$. To obtain that, we demonstrate that a large portion of these local phases cancel out and leave out a function close to a constant one.

Replacing $f_3$ by the function $\phi$ described in Theorem \ref{WSTEP: inverse theorem dla F}, we obtain the largeness of
$$
\sum_{x \in n\lfloor N^{1/3}\rfloor + [N^{1/3}]} \ \sum_{y \in [N^{1/3}]} \ \phi_0(x) \phi_1(x+y) \phi_2(x+2y) \phi(x+y^3)
$$
for many $n \in [N^{2/3}]$. Applying triangle inequality and expanding definitions of $\phi_i$-s we get
$$
\delta^{O(1)} N^{2/3} 
\ll
\sum_{y \in [N^{1/3}]} 
\left|
\sum_{x \in n\lfloor N^{1/3}\rfloor + [N^{1/3}]} 
e\big((\alpha_{0,n} +\alpha_{1,n} + \alpha_{2,n})x\big) \phi(x+y^3)
\right|.
$$
Observing that $\phi$ behaves like a constant function on intervals of length $N^{1/3}$ and using Weyl's inequality we obtain
$
\norm{\alpha_{0,n} +\alpha_{1,n} + \alpha_{2,n}} \ll \delta^{-O(1)}/N^{1/3}
$,
which is exactly the cancellation of phases that we need. Coming back to the original sum and choosing appropriate $y_0 \in [N^{1/3}]$ we finally obtain
\begin{multline*} 
\delta^{O(1)} N \ll
\left|
\sum_{x \in [N]} \spc
	\phi_0({x}) \phi_1({x} + y_0) \phi_2({x} + 2y_0) f_3({x} + y_0^3)
\right|
\leq
\\
\sum_{n \in [N^{2/3}]} \left|
\sum_{x \in n\lfloor N^{1/3} \rfloor+[N^{1/3}]} 
	e\big((\alpha_{0,n} +\alpha_{1,n} + \alpha_{2,n})x\big) f_3({x} + y_0^3)
\right|
\end{multline*}
which is sufficient to establish Theorem \ref{WSTEP: inverse thm dla ostatniej funkcji} due to the obtained cancellation of phases.

\subsubsection{Deriving density increment from the inverse theorems} The usual way of deducing density increment from inverse theorems is to express the counting operator as a main term and many error terms which result from expanding equality
$$
\Lambda_{y,2y,y^3}(1_A,1_A,1_A,1_A)
=\Lambda_{y,2y,3y^2}(f_A+\alpha 1_{[N]},f_A+\alpha 1_{[N]},f_A+\alpha 1_{[N]},f_A+\alpha 1_{[N]}),
$$
where $f_A = 1_A - \alpha1_{[N]}$ is the balanced indicator function and each of the error terms contains $f_A$ in at least one place. Then from the largeness of the main term and the lack of studied configuration inside $A$ we deduce the largeness of one of the error terms and apply an appropriate inverse theorem. This approach works when all inverse theorems are $U^1$-inverse theorems  \cite{PP19, Pel20}, but it fails in our case as the inverse theorems for the first three functions are $U^2$-inverse theorems.

We therefore need to modify the standard approach described above. We split our counting operator into only two parts, namely
$$
\Lambda_{y,2y,y^3}(1_A,1_A,1_A,1_A)
=
\Lambda_{y,2y,y^3}(1_A,1_A,1_A,\alpha1_{[N]}) +  \Lambda_{y,2y,y^3}(1_A,1_A,1_A,f_A).
$$
Since the first term counts arithmetic progressions of difference bounded by $N^{1/3}$, we apply Varnavides' supersaturation result for 3AP's \cite{Varnavides1959Certain} \cite[Theorem 2.1]{RimanicWolf2019Szemeredi} to obtain the largeness of this term. Employing the inverse theorem for the last term (Theorem \ref{WSTEP: inverse thm dla ostatniej funkcji}) finally allows us to deduce Theorem \ref{WSTEP: density increment}.

\section*{Acknowledgements}
We would like to thank Borys Kuca for numerous helpful suggestions. All authors of this paper were partially supported by NCN Polonez Bis 3 grant No. 2022/47/P/ST1/00854 (H2020 MSCA GA No. 945339). During the preparation of this article Ivan Spyrydonov was a participant of the tutoring programme under the Excellence Initiative at the Jagiellonian University.

\section{Results regarding Gowers' and box norms}
We start by recalling basic results about the $U^2$-norm, the proof of which is analogous to the proof of \cite[Lemma 2.2]{Go01}.
\begin{thm}[Inverse theorem for the $U^2$ norm]\label{Inverse theorem for the $U^2$-norm in Z} 
 Let $f\colon \Z \rightarrow \C$ be a 1-bounded function supported on $[N]$. Then there exists $\alpha \in \T$ such that
    $$ \norm{f}_{U^2}^{4} \leq N \abs{\sum_{x} f(x)e(\alpha x)}^2$$
\end{thm}

We also state a lemma that is the converse of the one stated above.

\begin{lem}\label{LEMMA KORELACJA Z FAZĄ DAJE DUŻĄ NORME U 2}
  Assume that $f\colon \Z \rightarrow \C$ is 1-bounded, supported in $[N]$, and satisfies 
  $$ \abs{\sum_x f(x)e(\alpha x)} \geq \delta N$$
  for some $\alpha \in \T$. Then
  $$
  \norm{f}_{U^2([N])}^4 \gg \delta^4 N^3.
  $$
\end{lem}
Now we state a result, which will be used later during our degree-lowering procedure.

\begin{lem}\cite[Lemma 3.5]{KKL24}\label{LEM: KKL24, properties of box-norms}
    Let $\delta > 0$, let $s \in \bb{N}$, and let $E_1,\ldots,E_s \subseteq G$ be finite nonempty sets. Let $f\colon \Z \rightarrow \bb{C}$ be $1$-bounded and supported on a finite set $B$. Then the following properties hold:
    \begin{enumerate}[(i)]
        \item (Inductive formula) For any $k \in [s]$, we have:
        $$
        {\lVert f \rVert}_{\Box(E_1,\ldots,E_s)}^{2^s} = \bb{E}_{h_1,h'_1 \in E_1} \cdots \bb{E}_{h_k,h'_k \in E_k} { \Big{\lVert} \Delta'_{(h_1,h'_1),\ldots,(h_k,h'_k)} f \Big{\rVert} }_{\Box(E_{k+1},\ldots,E_s)}^{2^{s-k}}.
        $$
        \item (Enlarging the sets) If $s \geq 2$ and $E_i' \subseteq \Z$ is a set satisfying $E_i \subseteq E_i'$ for each $i \in [s]$, then 
        $$
        {\lVert f \rVert}_{\Box(E_1,\dots,E_s)} \leq {\bigg( \frac{\abs{E_1'} \cdots '\abs{E_s'}}{\abs{E_1} \cdots \abs{E_s}} \bigg)}^{1/2^{s-1}}  {\lVert f \rVert}_{\Box(E_1',\dots,E_s')}.
        $$
        \item (Trimming the length of boxes) Let 
        $$
        E_i = \beta_{i} \cdot [\pm H_{i}] \quad \text{and} \quad E_i' = \beta_{i} \cdot [\pm \kappa_{i} H_{i}]
        $$
        for some $\beta_{i} \in \Z$, $\kappa_{i} > 0$, and $H_{i} \in \bb{N}_0$. Then:
        $$
         {\lVert f \lVert}_{        
    { \Box(E_1,\ldots, E_s)} 
            }^{2^{s}} \geq \delta \abs{B} \quad \Longrightarrow \quad {\lVert f \lVert}_{
    { \Box(E_1',\ldots, E_s')}
            }^{2^{s}} \spc {\gg}_{s} \spc \delta^{2^s} {\mathrm{min}(1,\delta^{2^s}/\kappa)}^{2s} \abs{B}
        $$
        for $\kappa = {\mathrm{max}}_{i j}(\kappa_{i j})$.

        \item (Passing to sub-arithmetic progressions) Suppose that:
        $$
         E_i = \beta_{i} \cdot [\pm H_{i}] \quad \text{and} \quad E_i' = q_{i} \beta_{i} \cdot [\pm H_{i}/q_{i}]
        $$
        for some $\beta_{i} \in \Z$ and $q_{i}, H_{i} \in \bb{N}$ satisfying $q_{i} \ll_{s} \delta^{2^s}H_{i}$ for all $i$. Then
        $$
        {\lVert f \lVert}_{        
    { \Box(E_1,\ldots, E_s)} 
            }^{2^{s}} \geq \delta \abs{B} 
            \quad 
            \Longrightarrow 
            \quad
            {\lVert f \lVert}_{        
    { \Box(E'_1,\ldots, E'_s)} 
            }^{2^{s}} \gg_{s} \delta \abs{B} .
        $$
    \end{enumerate}
\end{lem}

It will be important for us to understand the close relation between Gowers norms and box norms. The following result is standard.

\begin{lem} \label{OBS: Gowers-norm and box-norm equivalence}
    Let $f\colon \bb{Z} \rightarrow \C$ be a $1$-bounded function supported on $[N]$. Then for any integer $s \geq 2$, the following implications hold
    \begin{enumerate}[(i)]
        \item
        $
        {\lVert f \lVert}_{        
    { \Box([\pm N]) }^{s} 
            }^{2^{s}} \geq \delta N
        \quad
        \Rightarrow
        \quad
        {\lVert f \lVert}_{ U^s }^{2^{s}} \gg_s \delta N^{s+1}.
        $
        \item
        $
        {\lVert f \lVert}_{ U^s }^{2^{s}} \geq \delta N^{s+1}
        \quad
        \Rightarrow
        \quad
        {\lVert f \lVert}_{   
    { \Box([\pm N]) }^{s} 
            }^{2^{s}} \gg_s \delta^{2^s} N.
        $
    \end{enumerate}
\end{lem}

In our argument, we need to consider how polynomial pattern change when we restrict ourselves to arithmetic progressions. The following concept of $(C,q)$-coefficients was defined by Peluse to make such technicalities easier.

\begin{defi}\cite[Definition 3.1]{Pel20}
    A polynomial $P(y) = a_dy^d + \cdots + a_1y$ from $\Z[y]$ has
    $(C,q)$-coefficients if $|a_i| \leq C|a_d|$ for all $i \in [d-1]$ and $a_d = a'_dq^{d-1}$ with $0 <|a'_d| \leq C$.
\end{defi}

Before we end this section we present a less elementary, but crucial result. It follow from the PET induction scheme, which apriori outputs a control of an average of many local box norms, and then needs additional work to concatenate those. Remarkably more general version of the following result was obtained by Kravitz, Kuca and Leng in \cite{KKL24}.

\begin{thm} [Global box-norm-control] \cite[Theorem 1.6]{KKL24} \label{THM: global box-norm control for our config}
   There exists a positive integer $s = O(1)$ such that the following holds. If $N,M,q \in \N$ and $C,\delta > 0$ satisfy $\delta^{-O(1)} q^2 \spc {\ll} \spc N$ and
   $M = \frac{1}{2}\lfloor\sqrt[3]{N/q^2} \rfloor$, then for all 1-bounded functions $f_0, f_1, f_2, f_3 \colon \bb{Z} \rightarrow \bb{C}$ supported on $[N]$ and any cubic polynomial $P$ with $(C,q)$-coefficients and the leading term $q^2$, the bound
    $$
    \Bigg{|} \sum_{x } \sum_{y \in [M]} \spc
     f_0({x}) f_1({x} + y) f_2({x} + 2y) f_3\big({x} + P(y)\big)
    \Bigg{|} \geq \delta N M
    $$
    implies 
    $$
    {\lVert f_3 \lVert}_{        
    { \Box(q^2 \cdot [ \pm N/q^2]) }^{s} 
            }^{2^{s}} \spc {\gg}_C \spc \delta^{O_C(1)} N.
    $$
\end{thm}

\section{Degree lowering}

We start by citing the well known results which we will need to understand the behavior of our patterns counting operator when $f_3$ is a linear phase.

\begin{lem}\cite[Lemma 4.4]{QuantBehavPolNilManif}\label{LEM: Weyl's inequality}
Suppose that $g \colon   \Z \to \R$ is a polynomial of degree $d$ with leading coefficient $\alpha_d$, and that
$$ 
\abs{ \underset{n\in [N]}{\mathbb{E}}e(g(n))}\geq \delta
$$
for some $0<\delta<1/2$. Then there is $q\in \Z$, $
\abs{q}\ll \delta^{-O_d(1)}$ such that
$
\norm{q\alpha_d} \ll \delta^{-O_d(1)}/N^d
$.
\end{lem}

\begin{lem}\cite[Lemma 3.2]{QuantBehavPolNilManif} \label{LEM: concatenation of smallness}
    Let $\alpha \in \R$, $0 < \delta < 1/2$, $0 < \varepsilon \leq \delta/2$, and let $I \subseteq \R/\Z$ be an interval of length $\varepsilon$ such that $\alpha n \in I$ for at least $\delta N$ values of $n \in [N]$. Then there exists a positive integer $k \ll \delta^{-O(1)}$ such that $\norm{k\alpha} \ll \varepsilon \delta^{-O(1)}/N$.
\end{lem}

\begin{prop}
    \label{LEM: application of Weyl's inequality to our config}
    If $h_0, h_1, h_2$ are 1-bounded functions supported on the interval $[N]$, $\alpha \in \R$, $q,k \in \N$ and $P \in \Z[y]$ is a cubic polynomial with the main term $q^2$, where $d\in \N$ and such that
    \begin{align} \label{eq1}
    \abs{\sum_x \sum_{y \in [M]} h_0\big(q^kx-P(y)\big)h_1\big(q^kx+y-P(y)\big)h_2\big(q^kx+2y-P(y)\big)e(\alpha x) } \gg \delta M(N/q^k).
    \end{align}
    Then either $ \delta N \ll M$ or $M \ll  q^{k} \delta^{-O(1)}$, or there exists a positive integer $q' \ll \delta^{-O(1)}$ such that 
        $
        \norm{
            q'
            q^{
                2k+2
            }
            \alpha
        }
            \ll
        \delta^{-O(1)}{q^{
            3k
        }}/{M^{3}}
        $.
\end{prop}

\begin{proof}
    Substituting $y=v+q^{k}y'$, we obtain
  \begin{align*}
      q^kx-P(y) 
      = 
      q^kx-P(v+q^{k}y') 
      = 
      - P(v)+q^k
      \big(
        x-Q_v(y')
      \big) 
  \end{align*}
    for the cubic polynomial $Q_v(y) =
    \big(P(v+q^ky) - P(v)\big)/q^k \in \Z[y]$. Observe that for any $v \in [q^k]$, the polynomial $Q_v$ is a cubic and has the leading term $q^{2k+2}$. Substituting that formula in \eqref{eq1} we obtain

  \begin{multline*}
      \sum_{x} \sum_{y \in [M]} 
      h_0 \big( q^kx-P(y) \big) 
      h_1 \big( q^kx+y-P(y) \big)
      h_2 \big( q^kx+2y-P(y) \big) 
      e(\alpha x)
    \\
    =
        \sum_{v\in [q^k]} \spc \sum_{y' \in [M/q^k]-1} \sum_{x} \prod_{j=0}^2 
        h_j
        \bigg(
            \big( jv-P(v) \big)
            +q^k
                \big(
                x-Q_v(y')
                \big)
            +jq^k y'
       \bigg)
       e(\alpha x) 
    \\
    =
      \sum_{v\in [q^k]} \spc \sum_{y' \in [M/q^k]-1} \sum_{x} \prod_{j=0}^2 
        h_j
        \bigg(
            \big( jv-P(v) \big)
            +q^k 
            (
            x + j y'
            )
       \bigg)
      e(\alpha x) 
      e
      \big(
        \alpha Q_v(y')
      \big).
  \end{multline*}
 
  By the triangle inequality and pigeonhole principle, we can find $v \in [q^k]$ such that
  $$
  \left\lvert
  \sum_{y' \in [M/q^k]-1} \sum_{x} 
      \prod_{j=0}^2 
        h_j
        \bigg(
            \big( jv-P(v) \big)
            +q^k 
            (
            x + j y'
            )
       \bigg)
      e(\alpha x) 
      e
      \big(
        \alpha Q_v(y')
      \big)
  \right\lvert
  \gg
  \delta MN/q^{2k}.
  $$

Note that $ q^{k} \delta^{-O(1)} \ll M$ implies 
that obtained inequality still holds 
when $y'$ varies over the set $[M/q^k]$ instead of $[M/q^k]-1$. 
Split the sum over $x$ into intervals of length $M/q^k$ and apply the pigeonhole principle to obtain an integer $m$ such that  

\begin{align*}
  \Bigg|
  \sum_{x,y' \in [M/q^k]}  
  e(\alpha x) 
  \prod_{j=0}^2
  h_{j}
  \big(
        (jv-P(v)+q^km)+q^k(x+jy')
  \big)
  e
  \big(
        \alpha Q_v(y')
  \big)
  \Bigg|
  \gg
  \delta (M/q^k)^2.
  \end{align*}

Using the substitution 
$$
g_j=\begin{cases}
    h_{j}(jv-P(v)+q^km+q^kx)e(1_{j=0}\alpha x) 
    \cdot
    1_{[(j+1)M/q^k]}(x), \quad & j=0,1,2 \\
    e\big(\alpha Q_v(x)\big)
    \cdot
    1_{[M/q^k]}(x), \quad & j=3,
\end{cases}
$$
 we obtain
$$
\abs{\sum_x \sum_{y \in [M/q^k]}g_0(x)g_1(x+y)g_2(x+2y)g_3(y)} \gg \delta(M/q^k)^2.
$$

Use the Cauchy-Schwarz inequality,  
definition of $g_{0}$, substitution $h=y'-y$, recall support of the function $\Delta_hg_3$, and the triangle inequality to obtain
\begin{align*} 
   \delta^{2}(M/q^k)^{3} & \ll(M/q^k)^{-1} \cdot \abs{\sum_x \sum_{y\in [M/q^k]} g_{0}(x)g_{1}(x+y)g_{2}(x+2y)g_{3}(y)}^{2} \\
 &  \leq 
   \sum_{x} \spc \sum_{y,y' \in [M/q^k]}g_{1}(x+y)g_{2}(x+2y)g_{3}(y) \overline{g_{1}(x+y')g_{2}(x+2y')g_{3}(y')} \\
  & =  \sum_{x} \sum_{y\in [M/q^k]}\sum_{h\in [1-y, M/q^k-y]}\Delta_{h}g_{1}(x+y)\Delta_{2h}g_{2}(x+2y)\Delta_{h}g_{3}(y).\\
  & \leq 
  \sum_{h\in [\pm M/q^k]} \abs{\sum_{x,y} \Delta_{h}g_{1}(x+y)\Delta_{2h}g_{2}(x+2y)\Delta_{h}g_{3}(y)}.
\end{align*}
Then, for a fixed h, orthogonality, H\"older inequality and support of functions $g_j$ give us 
\begin{align*}
    &
    \abs{\sum_{x,y }\Delta_{h}g_{1}(x+y)\Delta_{2h}g_{2}(x+2y)\Delta_{h}g_{3}(y)} \\
   =  &   \abs{\int_{\T}\widehat{\Delta_{h}g_{1}}(\theta)\widehat{\Delta_{h}g_{2}}(-\theta)\widehat{\Delta_{h}g_{3}}(\theta)}
   \leq 
     \norm{\widehat{\Delta_{h}g_{1}}}_{2}\norm{\widehat{\Delta_{h}g_{2}}}_{2}\norm{\widehat{\Delta_{h}g_{3}}}_{\infty} \leq (3M/q^k)\norm{\widehat{\Delta_{h}g_{3}}}_{\infty} .
\end{align*}

In summary, we have obtained
$$
\sum_{h \in [\pm M/q^k]} 
\norm{
    \widehat{\Delta_{h}g_{3}}
}_{\infty} 
\gg
\delta^{2}(M/q^k)^{2}.
$$

Define $I_h = [M/q^k] \cap ([M/q^k]-h)$. Observe that for every $h \in [\pm M/q^k]$ satisfying the bound $\norm{\widehat{\Delta_{h}g_3}}_{\infty} \gg \delta^2(M/q^k)$, the inequality $|I_h| \gg \delta^2(M/q^k)$ 
holds. 

Unpacking definitions, we get

\begin{align*}
    \norm{\widehat{\Delta_h g_3 }}_{\infty} & =
    \sup_{\xi \in \bb{T}}
    \left\{
    \abs{
        \sum_{y \in I_h}e\big(
    (Q_v(y)-Q_v(y+h))+\xi y\big)   
    }
    \right\}\\
    & =
    \sup_{\theta_h \in \bb{T}}
    \left\{
    \abs{
        \sum_{y \in I_h} e(-3q^{2k+2}h\alpha y^2+\theta_h y)
    }
    \right\}
     =
    \abs{\sum_{y \in I_h} e(P_h(y))},
\end{align*}
where $P_h$ is degree 2 polynomial with main coefficient $-3q^{2k+2}h$.

Since $\abs{I_h} \gg \delta^2 (M/q^k) \gg \delta^{-O(1)}$, by Weyl's inequality (Lemma \ref{LEM: Weyl's inequality}), we deduce the existence of $q'_h \ll \delta^{O(1)}$ for at least $\gg \delta^2(M/q^k)$ different $h \in [\pm M/q^k]$ such that
$$
\norm{3q'_h q^{2k+2} h \alpha} 
\ll
\delta^{-O(1)} \frac{q^{2k}}{M^2}.
$$
We finish by applying Lemma \ref{LEM: concatenation of smallness}.
\end{proof}

Before we proceed to prove our degree lowering results, we have to recall two lemmas which will be important ingredients in our proof.

\begin{lem}[Dual–difference interchange] \cite[Lemma 6.3]{PP19} \label{LEM: dual-difference interchange}
 Let $M,N \in \N$ and let $I$ be an interval of length $N$. For each $y \in [M]$, let $F_y\colon \bb{Z} \rightarrow \C$ be a $1$-bounded function supported on $I$. Let $q$ be a positive integer. Set
$$
F(x) := \E{y}{M} F_y(x).
$$
Then for any function $\phi\colon \Z^{s} \rightarrow \T$ and a set $\mathcal{H} \subset q\cdot{[\pm N/q]}^{s}$, we have
\begin{multline*}
{\left( N^{-1} {(N/q)}^{-s} \sum_{ \underline{h} \in \mathcal{H}} \Bigg{|} 
\sum_x \Delta_{ \underline{h} } F(x) e\big( \phi( \underline{h} )x \big) 
\Bigg{|}
\right)}^{2^s}
\spc 
 \\
\ll_{s} N^{-1} {(N/q)}^{-2s} \sum_{
\underline{h}^0, \underline{h}^1 \in \mathcal{H}
}
\Bigg{|}
\sum_x
\E{y}{M} \Delta_{ \underline{h}^0-\underline{h}^1 } F_y(x) e\big(\phi(\underline{h}^0;\underline{h}^1)x\big)
\Bigg{|},
\end{multline*}
where
$$
\phi(\underline{h}^0;\underline{h}^1) := \sum_{\omega \in 
{\{0,1\}}^{s} 
}
{(-1)}^{\abs{\omega}} \phi({ \underline{h} }^{\omega})
\quad \quad \text{and} \quad \quad 
{ \underline{h} }^{\omega} := \big( {h_1}^{\omega_1},\ldots,{h_s}^{\omega_s} \big).
$$
\end{lem}

  If $\phi$ is a function of $s-1$ variables, then by $\phi(h_1, \ldots, \hat{h}_i, \ldots, h_s)$ we denote %it with 
  \\ $\phi(h_1, \ldots, h_{i-1}, h_{i+1}, \ldots, h_s)$. We say that the function $\phi$ of $s$ variables is low rank if there exist functions $\phi_1, \ldots, \phi_s$ such that
  $$ \phi(h_1, \ldots, h_s) = \sum_{i=1}^s \phi_i(h_1, \ldots,\hat{h}_i, \ldots, h_s)$$
  for any choice of possible parameters $h_1, \ldots, h_s$.

\begin{lem}[\cite{PP19}, Low rank correlation implies lower degree] \label{LEM: [PP22] low rank correlation implies lower degree}
    Let $M,N \in \N$ and let \\
    $f\colon \Z \rightarrow \C$ be a $1$-bounded function with support in $[N]$. Then for $\phi_1,\ldots,\phi_s \colon \Z^{s-1} \rightarrow \T$ with $m \leq s$, we have
    $$
    \frac{1}{N^{s+1}} \sum_{h_1,\ldots,h_s } \Bigg{|}
    \sum_x \Delta_{h_1,\dots,h_s} f(x) \cdot e \bigg(
    \sum_{i =1}^{m} 
    \phi_{i}(h_1,\ldots,\hat{h_i},\ldots,h_s)x
    \bigg)
    \Bigg{|} \spc \ll_{s} \spc {\bigg( 
    \frac{
    {\lVert f \rVert}_{U^{s+1}}^{2^{s+1}}
    }{N^{s+2}}
    \bigg)}^{2^{-m-1}}.
    $$
\end{lem}

Now we are ready to obtain the degree lowering result which is the main result of this section.

\begin{prop}[Degree lowering] \label{LEM: degree lowering}
    Let $f_{0}, f_1, f_2\colon \Z \rightarrow \C$ be 1-bounded functions with support in $[N]$. Let $P \in \Z[y]$ be a cubic polynomial with $(C,q)$-coefficients and leading term $q^2$, $M:=\lfloor \sqrt[3]{N/q^2} \rfloor$ and define the dual function as follows
   $$
   F(x):=
   \E{y}{M}
   f_0\big(x-P(y)\big)
   f_1\big(x+y-P(y)\big)
   f_2\big(x+2y-P(y)\big)
   .
   $$
   If $s\geq 3$, $k \geq 0$, and 
   $$
   {\lVert F \lVert}_{        
    { \Box(q^k \cdot [\pm 
    4N/q^k]) }^{s} 
            }^{2^{s}} \geq \delta N,
   $$
    then either $N \spc {\ll}_{s,k,c} \spc {(q/ \delta)}^{O_{s,k}(1)}$ or
   $$
   {\lVert F \lVert}_{        
    { \Box(q^{
        3k+2
    } \cdot [\pm 
    4N/q^{
        3k+2 
    }]) }^{s-1} 
            }^{2^{s-1}} \spc {\gg}_{s} \spc \delta^{O_s(1)} N.
   $$
\end{prop}

\begin{proof}
    For the entirety of this proof we assume $N \spc {\gg}_{s,k,c} \spc {(q / \delta)}^{O_{s,k}(1)}$, as otherwise the first part of the conclusion is satisfied. Given $u\in [q^k]$, let 
    $$F_{u}(x):=F(u+q^k x).$$ 
    Observe that $F$ is supported on the interval $$[-2cq^2M^{2},2N + 2cq^2M^{2}],$$ hence from $N \spc {\gg}_{s,k,c} \spc {(q /\delta)}^{O_{s,k}(1)}$ we deduce that $F_u$ is supported on $[-N/q^k,3N/q^k]$. Applying the popularity principle we conclude that there exists a set of $\Omega(\delta q^k)$ residues $u\in [q^k]$ for which $${\lVert F_u \lVert}_{  
    { \Box([\pm 4N/q^k]) }^{s} 
     }^{2^{s}} \gg \delta N/q^k.$$     

     Now, since we can assume $q^{2k+2} \spc {\ll}_s \spc \delta^{2^s} 4N/q^k$, we apply Lemma \ref{LEM: KKL24, properties of box-norms} (iv) to obtain inequality
     $$
     {\lVert F_u \lVert}_{        
    { \Box(q^{2k+2} \cdot [\pm 4N/q^{3k+2}],\ldots, q^{2k+2} \cdot [\pm 4N/q^{3k+2}], [\pm 4N/q^k],[\pm 4N/q^k])} 
            }^{2^{s}} \gg_{s} \delta N/q^k.
     $$  

    Using inductive formula for the box-norm (cf. Lemma \ref{LEM: KKL24, properties of box-norms} (i)) we get
    \begin{multline*}
    \bb{E}_{h_1,h'_1 \in q^{2k+2} \cdot [\pm 4N/q^{3k+2}]} \cdots \bb{E}_{h_{s-2},h'_{s-2} \in q^{2k+2} \cdot [\pm 4N/q^{3k+2}]}
    \\
    {\Big{\lVert} \Delta'_{(h_1,h'_1),\ldots,(h_{s-2},h'_{s-2})} F_u \Big{\rVert} }_{\Box([\pm 4N/q^k], [\pm 4N/q^k])}^{4} \gg_s \delta (N/q^k).
    \end{multline*}
    
    Rewriting last inequality with the use of Fej\'er kernel we obtain
    \begin{align*}
  &  \delta (N/q^k)
    \\
   \ll_s
&    \sum_{h_1,\ldots,h_{s-2} } \spc \mu_{(q^{2k+2} \cdot [\pm 4N/q^{3k+2}])}(h_1,\ldots,h_{s-2}) \cdot
     { \Big{\lVert} \Delta_{h_1,\ldots,h_{s-2}} F_u \Big{\rVert} }_{\Box([\pm 4N/q^k], [\pm 4N/q^k])}^{4} 
     \\
   \ll_s
 &   \sum_{h_1,\ldots,h_{s-2} \in q^{2k+2} \cdot [\pm 4N/q^{3k+2}]}
    {(q^{3k+2}/N)}^{s-2}
    \cdot
     { \Big{\lVert} \Delta_{h_1,\ldots,h_{s-2}} F_u \Big{\rVert} }_{\Box([\pm 4N/q^k], [\pm 4N/q^k])}^{4}
     \\
     \ll_s
&     \bb{E}_{h_1 \in q^{2k+2} \cdot [\pm 4N/q^{3k+2}]} \cdots \bb{E}_{h_{s-2} \in q^{2k+2} \cdot [\pm 4N/q^{3k+2}]} { \Big{\lVert} \Delta_{h_1,\ldots,h_{s-2}} F_u \Big{\rVert} }_{\Box([\pm 4N/q^k], [\pm 4N/q^k])}^{4}.
    \end{align*}

    Using the popularity principle we conclude the existence of 
    $\mathcal{H} \subset q^{2k+2} [\pm 4N/q^{3k+1}]^{s-2}$ of size $\gg_s \delta {(N/q^{3k+2})}^{s-2}$, such that for any $\underline{h} \in \mathcal{H}$, the inequality
    $$
    { { \Big{\lVert} \Delta_{ \underline{h} } } F_u \Big{\rVert} }_{\Box([\pm 4N/q^k], [\pm 4N/q^k])}^{4} \gg_s \delta (N/q^k)
    $$
    holds. Applying the Lemma \ref{OBS: Gowers-norm and box-norm equivalence} (i) shows that for any $\underline{h} \in \mathcal{H}$, we have
    $$
    { \Big{\lVert} \Delta_{\underline{h}} F_u \Big{\rVert} }_{U^2}^{4} \gg_s \delta {(N/q^k)}^3.
    $$
    Hence, by Theorem \ref{Inverse theorem for the $U^2$-norm in Z}, we conclude the existence of $\phi( \underline{h} ) \in \T$ which satisfy
    $$
    \abs{
    \sum_x\Delta_{ 
    \underline{h}
    }F_u(x)
    e\big( 
    \phi( 
    \underline{h}
    )x
    \big)
    } \gg_s \delta N/q^k.
   $$
 
    This implies the existence of a function $\phi \colon \Z^{s-2} \rightarrow \T$ such that for every $\underline{h} \in \mathcal{H}$, we have
    \begin{align} \label{EQUA: corelation of F_u with a phase}
    \abs{
    \sum_x \Delta_{ 
    \underline{h}
    }F_u(x)
    e\big( 
    \phi( 
    \underline{h}
    )x
    \big)
    } \gg_s \delta N/q^k.
   \end{align}

Let $C_s$ be sufficiently large that for any $T' \geq C_{s}\delta^{-1}N/q^k$ rounding $\phi( \underline{h} )$ to the nearest fraction of the form $t'/T'$, the validity of (\ref{EQUA: corelation of F_u with a phase}) remains for every $\underline{h} \in \mathcal{H}$. We can additionally assume that $q^{2k+2} \leq C_{s}\delta^{-1}N/q^k$, as it follows from the assumption that $N \spc {\gg}_{s,k,c} \spc {(q /\delta)}^{O_{s,k}(1)}$ if we choose constants correctly in the last inequality.
 
We therefore conclude the existence of $T \in \Z$ which satisfy the inequalities
$$C_{s}\delta^{-1}N/q^k \leq T \leq 2C_{s}\delta^{-1}N/q^k$$ 
and is divisible by $q^{2k+2}$. In particular, it means that for each $\underline{h} \in \mathcal{H}$, rounding $\phi( \underline{h} )$ to the nearest fraction of the form $t/T$ preserves the validity of (\ref{EQUA: corelation of F_u with a phase}). For that reason we can assume that $$\mathrm{Im}(\phi) \subseteq \big\{ t/T \in \T : t \in \Z \big\}.$$

Summing over $\underline{h} \in \mathcal{H}$ and applying Lemma \ref{LEM: dual-difference interchange}, we deduce that

\begin{multline*}
    \sum_{
\underline{h}^0, \underline{h}^1 \in \mathcal{H}
}
\Bigg{|}
\sum_x
\E{y}{M} 
\Delta_{ \underline{h}^0-\underline{h}^1 } 
f_0\big(u+q^kx-P(y)\big)
\Delta_{ \underline{h}^0-\underline{h}^1 }
f_1\big(u+q^k x+y-P(y)\big)
\\
\Delta_{ \underline{h}^0-\underline{h}^1 }
f_2\big(u+q^k x+2y-P(y)\big)
e\big(\phi(\underline{h}^0;\underline{h}^1)x\big)
\Bigg{|} 
\gg_s
\delta^{O_s(1)} (N/q^k) {(N/q^{3k+2})}^{2(s-2)}.
\end{multline*}

Applying the pigeonhole and popularity principles, there exists $\mathcal{H'}\subset \mathcal{H}$ of size at least $\gg_s \delta^{O_s(1)}(N/q^{3k+2})^{s-2}$ and $\underline{h}^1 \in \mathcal{H}$ such that for every $\underline{h}^0 \in \mathcal{H}'$, we have

\begin{multline*}
    \Bigg{|}
\sum_x
\sum_{y \in [M]} 
\Delta_{ \underline{h}^0-\underline{h}^1 } 
f_0\big(u+q^kx-P(y)\big)
\Delta_{ \underline{h}^0-\underline{h}^1 }
f_1\big(u+q^k x+y-P(y)\big)
\\
\Delta_{ \underline{h}^0-\underline{h}^1 }
f_2\big(u+q^k x+2y-P(y)\big)
e\big(\phi(\underline{h}^0;\underline{h}^1)x\big)
\Bigg{|} 
\gg_s
\delta^{O_s(1)} M (N/q^k).
\end{multline*} 

By the Proposition \ref{LEM: application of Weyl's inequality to our config}, for each $\underline{h}^0 \in \mathcal{H}'$ there exists $q' \ll_s \delta^{-O_s(1)}$ such that
$$
\norm{
    q'q^{2k+2}
    \phi(\underline{h}^0;\underline{h}^1)
} 
\ll_s
\delta^{-O_s(1)}q^{3k+2}/N
.
$$

Notice that $q^{2k+2}
\phi(\underline{h}^0;\underline{h}^1)$ is an element of the additive group $$\{t/T_0: \, t\in [T_0]\}\subset \T,$$ 
where $T_0 := T/q^{2k+2}$. Moreover, for any positive numbers $Q_1,Q_2$ we have the inclusion

$$
\{\alpha \in \T: \exists q' \leq Q_1 \spc \text{with} \norm{q'\alpha}\leq Q_2 q^{3k+2}/N\} 
\subset
\bigcup_{
\substack{
1\leq a \leq q'\leq Q_1
}
}
\left[ \frac{a}{q'}-\frac{q^{3k+2}Q_2}{N}, \frac{a}{q'}+\frac{q^{3k+2}Q_2}{N}\right].
$$

By a volume packing argument, the number of fractions $t/T_0$ lying in this union of those intervals is at most
$
Q_1^2 (1+\frac{2q^{3k+2} Q_2 T_0}{N})
$, assuming $q^{3k+2}Q_2 \ll N$, hence if we set $Q_1$ and $Q_2$ to be the appropriate upper bounds obtained from Proposition \ref{LEM: application of Weyl's inequality to our config}, we conclude that 
$$
Q_1^2 \Bigg(1+\frac{2q^{3k+2} Q_2 T_0}{N}\Bigg) 
\ll_s
\delta^{-O_s(1)} 
+
\delta^{-O_s(1)} 
\frac{2q^k \delta^{-O_s(1)} T}{N} 
\ll_s
\delta^{-O_s(1)}.
$$

It therefore follows from the pigeonhole principle that there exists $\mathcal{H}''\subset \mathcal{H}'$ of size at least $\gg_s \delta^{O_s(1)}(N/q^{3k+2})^{s-2}$ and $t_0 \in [T_0]$ such that for any $\underline{h}^0 \in \mathcal{H}''$, we have 
$$q^{2k+2} \phi(\underline{h}^0;\underline{h}^1)=t_{0}/T_0.$$
In particular, when restricted to the set $\mathcal{H}''$, the function $q^{2k+2}\phi(\cdot)$ satisfies
\begin{align}\label{EQUA: q^3 phi is low rank on a large set}
q^{2k+2}\phi(\underline{h}^0)=t_0/T_0 \spc - \sum_{\omega \in \{0,1\}^{s-2} \setminus \{0\}}(-1)^{\abs{\omega}} q^{2k+2}\phi(\underline{h}^{\omega}).
\end{align}

The right-hand side of this equality is low rank according to the terminology preceding Lemma \ref{LEM: [PP22] low rank correlation implies lower degree}. Summing over  $\underline{h} \in \mathcal{H}''$ in (\ref{EQUA: corelation of F_u with a phase}) and applying the triangle inequality, 
we obtain:

\begin{multline*}
\delta^{O_s(1)} (N/q) {(N/q^{3k+2})}^{s-2}
\ll_s
\sum_{
\underline{h} \in \mathcal{H}''
}\abs{\sum_x \Delta_{\underline{h}}  F_u(x)e(\phi(\underline{h})x)} 
\\
\leq
\sum_{v \in [q^{2k+2}]}
\sum_{
\underline{h} \in \mathcal{H}''
}
\abs
{
\sum_x \Delta_{\underline{h}}  F_u(v+q^{2k+2} x)
e
\big(
q^{2k+2}\phi(\underline{h})\cdot x
\big)
}
\end{multline*}

Defining $\widetilde{F}_{u,v}$ by the formula $\widetilde{F}_{u,v} (x) := F(u+q^kv+q^{3k+2}x)$ and applying the popularity principle allows us to conclude the existence of at least $\gg_s \delta^{O_s(1)}q^{2k+2}$ residues $v \in [q^{2k+2}]$ which satisfy

$$
\delta^{O_s(1)} {(N/q^{3k+2})}^{s-1}
\ll_s
\sum_{
\substack{ 
\underline{h}' \\
q^{2k+2} \underline{h}' \in \mathcal{H}''
}
}
\abs
{
\sum_x \Delta_{\underline{h}'}  \widetilde{F}_{u,v}(x)
e
\big(
q^{2k+2}\phi(q^{2k+2}\underline{h}')
\cdot x
\big)
}.
$$

Now thanks to equation \eqref{EQUA: q^3 phi is low rank on a large set} we see that the function $\mathcal{H}'' \ni \underline{h}' \mapsto q^{2k+2}\phi(q^{2k+2}\underline{h}') \in \T$ is low-rank, hence it can be extended to a low rank function $\psi \ \colon \ \Z^{s-2} \rightarrow \T$ satisfying 
$$
\delta^{O_s(1)} {(N/q^{3k+2})}^{s-1}
\ll_s
\sum_{
\substack{ 
\underline{h}' \in \Z
}
}
\abs
{
\sum_x \Delta_{\underline{h}'}  \widetilde{F}_{u,v}(x)
e
\big(
\psi(\underline{h}')
\cdot x
\big)
}.
$$

Employing Lemma \ref{LEM: [PP22] low rank correlation implies lower degree} gives 
$$
\norm{
\widetilde{F}_{u,v}
}^{2^{s-1}}_{U^{s-1}} \gg_s \delta^{O_s(1)}(N/q^{3k+2})^s.
$$

Applying the Lemma \ref{OBS: Gowers-norm and box-norm equivalence} (ii) we conclude
$$
\norm{
\widetilde{F}_{u,v}
}^{2^{s-1}}_{ \Box{([\pm 4N/q^{3k+2}])}^{s-1} } \gg_s \delta^{O_s(1)}(N/q^{3k+2}).
$$

Summing over permissible $v \in [q^{2k+2}]$ and extending by nonnegativity we obtain 
$$
\delta^{O_s(1)}(N/q^k) \ll_s \sum_{v \in [q^{2k+2}]} 
\norm{
\widetilde{F}_{u,v}
}^{2^{s-1}}_{ \Box{([\pm 4N/q^{3k+2}])}^{s-1} } =
\norm{
{F}_{u}
}^{2^{s-1}}_{ \Box{(q^{2k+2} \cdot [\pm 4N/q^{3k+2}])}^{s-1} }.
$$

To end the proof we perform similar procedure with respect to $u \in [q^k]$. We sum over permissible $u \in [q^k]$ and then extend by nonnegativity
$$
\delta^{O_s(1)}N \ll_s \sum_{u \in [q^k]} 
\norm{
{F}_{u}
}^{2^{s-1}}_{ \Box{(q^{2k+2}\cdot [\pm 4N/q^{3k+2}])}^{s-1} } =
\norm{
{F}
}^{2^{s-1}}_{ \Box{(q^{3k+2} \cdot [\pm 4N/q^{3k+2}])}^{s-1} }.
$$
\end{proof}

\section{$U^1$-control of the dual function $F$}

We are now ready to prove an important result regarding the structure of the dual function.

\begin{thm}[$U^1$-control of the dual $F_3$] \label{THM: U^1 control of F_3}
If $N, q \in \N$ and $\delta,c > 0$ satisfy ${(q/\delta)}^{O(1)} \ll_c N$, then for all $1$-bounded functions $f_0, f_1, f_2, f_3 \ \colon \ \bb{Z} \rightarrow \bb{C}$ supported on $[N]$ and any cubic polynomial $P \in \Z[y]$ which have $(C,q)$-coefficients and the main term $q^2$, the bound
	$$
	\Bigg{|} \sum_{x } \sum_{y \in \N} \spc
	f_0({x}) f_1({x} + y) f_2({x} + 2y) f_3\big({x} + P(y)\big)
	\Bigg{|} \geq \delta N M,
	$$
	where $M = \lfloor \sqrt[3]{N/q^2} \rfloor$, implies the existence of positive integers $q' \ll_c \delta^{-O_c(1)}$, $K=O(1)$ and the function $\phi\colon \Z \rightarrow \C$ 
    such that for each $u \in [q'q^K]$ there exist $\alpha_u, \beta_u \in \T$ such that $\phi(u + q'q^K x) = e(\alpha_u x + \beta_u)$, $\norm{\alpha_u} \ll_c {\delta}^{-O_c(1)}q^K/N$ and moreover $\phi$ satisfies the inequality
	$$
	\sum_x F(x) \phi(x) \gg_c \delta^{O_c(1)} N,
	$$
	where 
    $$
    F(x):=\underset{y \in [M]}{\bb{E}}
    f_0 \big(x-P(y)\big)
    f_1 \big(x+y-P(y)\big)
    f_2 \big(x+2y-P(y)\big)
    .
    $$
    In particular, the function $\phi$ is $\ll_c {\delta}^{-O_c(1)}q^K/N$ Lipschitz along arithmetic progressions of common difference $q'q^K$.
\end{thm}

\begin{proof}
    Observe that if $N \gg_c {(q/\delta)}^{O(1)}$, then for any $y \geq (1 + \delta/2)M$ the inequality $P(y) \geq N$ holds as a consequence of the fact that $P$ has $(C,q)$-coefficients, hence we deduce that
    $$
	\Bigg{|} \sum_{x } \sum_{y \in [M]} \spc
	f_0({x}) f_1({x} + y) f_2({x} + 2y) f_3\big({x} + P(y)\big)
	\Bigg{|} \geq (\delta/2) N M.
	$$
    Let us now observe that 
    $$
    \sum_{x } \sum_{y \in [M]} \spc
	f_0({x}) f_1({x} + y) f_2({x} + 2y) f_3\big({x} + P(y)\big) 
    =
    M \sum_x F(x) f_3(x),
	$$
    hence by the Cauchy-Schwarz inequality we obtain
    $$
    \delta^2 NM \ll M \sum_x F(x) \overline{F(x)} = \sum_{x} \sum_{y \in [M]} \spc
	f_0({x}) f_1({x} + y) f_2({x} + 2y) \overline{F({x} + P(y))}.
    $$

    Observe that as long as ${(q/\delta)}^{O(1)} \ll_c N$ all of the above functions are supported inside the interval $[\pm4N]$ and
    $$
    \Big| 2M -\sqrt[3]{8N/q^2} \Big| \leq 2,
    $$
    hence we can apply Theorem \ref{THM: global box-norm control for our config} to obtain the inequality
    $$
    {\big\lVert \overline{F} \big\lVert}_{        
    { \Box(q^2 \cdot [ \pm 8N/q^2]) }^{s} 
            }^{2^{s}} \gg_c \delta^{O_c(1)} N
    $$
    for some $s= O(1)$. Employing Lemma \ref{LEM: KKL24, properties of box-norms} (iii) we obtain
    $$
    {\big\lVert \overline{F} \big\lVert}_{        
    { \Box(q^2 \cdot [ \pm 4N/q^2]) }^{s} 
            }^{2^{s}} \gg_c \delta^{O_c(1)} N.
    $$

    After observing that conjugation does not change the value of the box norm and applying the Proposition \ref{LEM: degree lowering} $s-2$ times we obtain the inequality
    $$
   {\lVert F \lVert}_{        
    { \Box(q^{a(s-2)} \cdot [\pm 4N/q^{a(s-2)}]) }^{2} 
            }^{4} \gg_c \delta^{O_c(1)} N,
   $$
    where $a\colon \N \rightarrow \N$ is given by the inductive formula $a(n+1) = 3a(n)+2$ and starts from $a(0)=2$. Since $s = O(1)$, we also have $k=a(s-2) = O(1)$. Now observe that
   $$
    \sum_{u \in [q^k]} {\lVert F_u \lVert}_{        
    { \Box([\pm 4N/q^{k}]) }^{2} 
            }^{4} = {\lVert F \lVert}_{        
    { \Box(q^{k} \cdot [\pm 4N/q^{k}]) }^{2} 
            }^{4} 
    \gg_c
    \delta^{O_c(1)} N,
   $$
   where $F_u(x) := F(u+q^k x)$. By popularity principle we find the set $\mathcal{U} \subseteq [q^k]$ of size \\
   $\abs{\mathcal{U}} \gg_c \delta^{O_c(1)}q^k$ such that each element $u \in \mathcal{U}$ satisfies inequality
   $$
   {\lVert F_u \lVert}_{        
    { \Box([\pm 4N/q^{k}]) }^{2} 
            }^{4}
    \gg_c 
    \delta^{O_c(1)} (N/q^k).
   $$
   Combining Lemma \ref{OBS: Gowers-norm and box-norm equivalence} with Theorem \ref{Inverse theorem for the $U^2$-norm in Z}, we deduce that there exists $\gamma_u \in \T$ such that
    $$
    \Bigg|
    \sum_x
    F_u(x) e(\gamma_u x)
    \Bigg|
    \gg_c
    \delta^{O_c(1)} (N/q^k).
   $$
   Expanding the definition of $F_u$ we see that 
\begin{multline*}
\Bigg{|}
\sum_x
\sum_{y \in [M]}
f_0\big(u+q^kx-P(y)\big)
f_1\big(u+q^k x+y-P(y)\big)
\\
f_2\big(u+q^k x+2y-P(y)\big)
e\big(\gamma_u x\big)
\Bigg{|} 
\gg_c
\delta^{O_c(1)} (N/q^k) M. 
\end{multline*}
By Proposition \ref{LEM: application of Weyl's inequality to our config} we see that there exist positive integers $q_{u} \ll_c \delta^{-O_c(1)}$ such that $$\norm{q_u q^{2k+2}\gamma_u} \ll_c \delta^{-O_c(1)}{q^{3k}}/{M^{3}}.$$

By the pigeonhole principle, there exists $q' \ll_c \delta^{-O_c(1)}$ and a set $\mathcal{U'} \subset \mathcal{U}$ still satisfying $\abs{\mathcal{U'}} \gg_c \delta^{O_c(1)}q^k$ such that for all $u \in \mathcal{U'}$, we have
$$
\norm{q' q^{2k+2}\gamma_u} 
\ll_c
\delta^{-O_c(1)}{q^{3k}}/{M^{3}} 
\ll
\delta^{-O_c(1)} q^{3k+2}/N
.
$$

Now that $q'$ has been chosen let $K=3k+2$. For any $u \in [q'q^K]$ let
\begin{align}\label{definition of alpha u}
\alpha_u = \begin{cases}
    \gamma_{u \ (\text{mod} \ q^k)} & \text{if } u \ (\text{mod} \ q^k) \in \mathcal{U}'; \\
    0 & \text{otherwise};
\end{cases}
\end{align}
and let $\beta_u$ be such that
$$
e(\beta_u) \cdot \sum_{x } F(u + q'q^K x)e(\alpha_u x) \geq 0.
$$

It is now easy to observe that the function $\phi \ \colon \ \Z \rightarrow \C$ given by the formula
$$
\phi(u + q'q^K x) = e(\alpha_u x + \beta_u) 
$$
satisfies all desired properties.
\end{proof}

\section{Inverse theorem for $f_0$, $f_1$ and $f_2$}
We start by recalling standard result. One example of the reference is \cite[Lemma 3.1]{Pre17}.

\begin{lem}[Van der Corput inequality]\label{van der Corput inequality} 
    Let $f\colon \Z \rightarrow \C$ be 1-bounded, $I$ be an interval containing at most $M$ integers, f is supported on $I$, and $H\geq 1$. For each $h \in [\pm H]$, there exists an interval $I(h)\subset I$ such that
    $$
    \abs{\sum_{y \in I}f(y)}^2\leq (M+H)\sum_{h \in [\pm H]} \mu_H(h)\sum_{y \in I(h)}\Delta_h f(y).
    $$
\end{lem}

Before we prove the main theorem of this section, for convenience we first show a known result which will be needed in the main proof of the main theorem.

\begin{lem} \label{LEM: niepełny operator zliczający 3AP daje korelacje z fazą liniową}
    Let $M \in \N$, $\delta >0$, and $f_0,f_1,f_2 \ \colon \ \Z \rightarrow \C$ be $1$-bounded functions such that $f_j$ is supported on the interval $[(j+1)M]$. Suppose that $f_0, f_1 \ \text{and} \ f_2$ satisfy
    $$
    \Bigg|
    \sum_{x \in [M]} \sum_{y \in [M]} f_0(x) f_1(x+y) f_2(x+2y)
    \Bigg|\geq \delta M^2
    .
    $$
 Then for all $i \in \{0,1,2\}$ the inequality $ {\big\lVert \widehat{f_i} \big\rVert}_{\infty} \gg \delta^2 M$ holds.
\end{lem}

\begin{proof}
Since $f_0$ is supported on $[M]$, we may write
\[
\left| \sum_{x\in[M]}\sum_{y\in[M]} f_0(x)f_1(x+y)f_2(x+2y) \right| =
\left| \sum_x\sum_{y\in[M]} f_0(x)f_1(x+y)f_2(x+2y) \right|.
\]
By the Cauchy-Schwarz inequality, the $1$-boundedness of $f_0$, and Lemma~\ref{van der Corput inequality},
\[
\begin{aligned}
\left| \sum_x\sum_{y\in[M]} f_0(x)f_1(x+y)f_2(x+2y) \right|^2
&\le M\sum_x \left| \sum_{y\in[M]}f_1(x+y)f_2(x+2y) \right|^2  \\
&\ll M^2\sum_{h_1\in[\pm M]}\mu_M(h) \sum_x\sum_{y\in I(h_1)} \Delta_{h_1}f_1(x+y)\Delta_{2h_1}f_2(x+2y),
\end{aligned}
\]
where $I(h)=[M]\cap ([M]-h)$. After replacing $x$ by $x-y$, the last expression becomes
\[
M^2\sum_{h_1\in[\pm M]}\mu_M(h_1) \sum_x\Delta_{h_1}f_1(x) \sum_{y\in I(h_1)}\Delta_{2h_1}f_2(x+y).
\]
Since $\sum_{h_1}\mu_M(h_1)=1$, another application of the Cauchy-Schwarz inequality yields
\[
\begin{aligned}
\left|
\sum_x\sum_{y\in[M]} f_0(x)f_1(x+y)f_2(x+2y) \right|^4 \ll
M^4\sum_{h_1\in[\pm M]}\mu_M(h_1) \left| \sum_x\Delta_{h_1}f_1(x) \sum_{y\in I(h)}\Delta_{2h_1}f_2(x+y) \right|^2.
\end{aligned}
\]
As $f_1$ is supported on $[2M]$, the Cauchy-Schwarz inequality in $x$, followed by another application of Lemma~\ref{van der Corput inequality}, gives
\[
\begin{aligned}
\left| \sum_x\Delta_{h_1}f_1(x) \sum_{y\in I(h_1)}\Delta_{2h_1}f_2(x+y) \right|^2
&\ll M\sum_x \left|\sum_{y\in I(h_1)}\Delta_{2h_1}f_2(x+y)\right|^2 \\
&\ll M^3\sum_{h_2\in[\pm M]}\mu_M(h_2) \sum_x\Delta_{2h_1,h_2}f_2(x).
\end{aligned}
\]
Consequently,
\[
\begin{aligned}
\left| \sum_x\sum_{y\in[M]} f_0(x)f_1(x+y)f_2(x+2y) \right|^4
&\ll M^7 \sum_{h_1,h_2\in[\pm M]} \mu_M(h_1)\mu_M(h_2) \sum_x\Delta_{2h_1,h_2}f_2(x) \\
&= M^7\|f_2\|_{\Box(2[\pm M],[\pm M])}^4 \ll M^7\|f_2\|_{\Box([\pm 2M])^2}^4,
\end{aligned}
\]
where the last inequality follows from Lemma~\ref{LEM: KKL24, properties of box-norms}~(ii).

By Lemma~\ref{OBS: Gowers-norm and box-norm equivalence}~(i), we therefore obtain
\[
\|f_2\|_{U^2}^4\gg\delta^4M^3.
\]
Since $f_2$ is supported on an interval of length $3M$, Theorem~\ref{Inverse theorem for the $U^2$-norm in Z} gives $\|f_2\|_{U^2}^4 \ll M\|\widehat f_2\|_\infty^2$. Thus
\[
\|\widehat f_2\|_\infty\gg\delta^2M.
\]

The arguments for $f_0$ and $f_1$ are analogous. This proves the result for every $i\in\{0,1,2\}$.
\end{proof}

\begin{thm}[Inverse theorem for $f_0$,$f_1$ and $f_2$] \label{THM: inverse theorem for the first three functions stated as the local gowers norm control}
If $N,q \in \N$ and $\delta > 0$ satisfy  ${(q/\delta)}^{O(1)} \ll N$, then for all 1-bounded functions $f_{0}, f_{1}, f_{2}, f_{3}\ \colon \ \Z \rightarrow \C$ supported on $[N]$, the bound 
    $$
    \abs{\sum_{x} \sum_{y\in \N} f_{0}(x)f_{1}(x+y)f_{2}(x+2y)f_{3}(x+q^2y^{3})} 
    \gg
    \delta N M
    ,
    $$
    where $M = \lfloor \sqrt[3]{N/q^2} \rfloor$, implies the existence of positive integers $q' \ll \delta^{-O(1)}$, $K = O(1)$ and $M' \gg
    \delta^{O(1)} M/q^K
    $ such that for each $i \in \{0,1,2\},$ 
    $$
    \sum_x \norm{f_i}^{4}_{U^2(x+q'q^K\cdot[M'])} \gg \delta^{O(1)} N {M'}^3.
    $$
\end{thm}

\begin{proof}
    We start by applying Theorem \ref{THM: U^1 control of F_3} for polynomial $P(y)=q^2y^3$, which has \\
    $(1,q)$-coefficients, to obtain
    $$
    \sum_x \sum_{y \in [M]} f_0(x-q^2y^3) f_1(x+y-q^2y^3) f_2(x+2y-q^2y^3) \phi(x) \gg \delta^{O(1)} NM
    $$
    for $\phi$ which is $\ll {\delta}^{-O(1)}q^K/N$ Lipschitz along arithmetic progressions of common difference $q'q^K$ for some $q' \ll \delta^{-O(1)}$ and $K=O(1)$. Recall that in fact for any $u \in [q'q^K]$ it has the form $\phi(u+q'q^Kx)=e(\alpha_ux+\beta_u)$ with $\alpha_u$ defined as in \ref{definition of alpha u}. Shifting $x$ to $x+q^2y^3$ we obtain
    $$
    \Bigg|
    \sum_{x \in [N]} \sum_{y \in [M]} f_0(x) f_1(x+y) f_2(x+2y) \phi(x+q^2y^3)
    \Bigg|
    \gg
    \delta^{O(1)} NM.
    $$
    Using the triangle inequality and popularity principle we find a subset $\mathcal{U} \subset {[q'q^K]}^2$ such that $\abs{ \mathcal{U}} \gg \delta^{O(1)}{(q'q^K)}^2$, and for each $(u,v) \in \mathcal{U}$, we have
    \begin{multline*}
    \Bigg|
    \sum_{x \in [N/(q'q^K)]} \sum_{y \in [M/(q'q^K)]} 
    f_0(u+q'q^Kx) 
    f_1(u+q'q^Kx+v+q'q^Ky)
    f_2\big(u+q'q^Kx+2(v+q'q^Ky)\big) 
    \\
    \phi\big(u+q'q^Kx+q^2{(v+q'q^Ky)}^3\big)
    \Bigg|
    \gg
    \delta^{O(1)} (N/(q'q^K))(M/(q'q^K)).
    \end{multline*}

    Substitute
    \begin{align*}
        w & = u + q^2v^3 \\
        P_{u,v}(x) & = {q'}^{2}q^{2K+2} x^3 + 3q'q^{K+2}v x^2 + 3v^2q^2x \\
        g_{0,u,v}(x) & = f_0(u+q'q^Kx) e(\alpha_w x) \\
        g_{1,u,v}(x) & = f_1(u+v+q'q^Kx) \\
        g_{2,u,v}(x) & = f_2(u+2v+q'q^Kx) \\
        g_{3,u,v}(x) & = e\big( \alpha_w P_{u,v}(x) \big),
    \end{align*}
    
 where $\alpha_{w}$ is defined as in \eqref{definition of alpha u}. Then the previous inequality takes the form
    $$
    \Bigg|
    \sum_{x \in [N/(q'q^K)]} \sum_{y \in [M/(q'q^K)]} 
    g_{0,u,v}(x) 
    g_{1,u,v}(x+y)
    g_{2,u,v}(x+2y)
    g_{3,u,v}(y)
    \Bigg|
    \gg
    \delta^{O(1)} (N/(q'q^K))(M/(q'q^K)).
    $$
    Since $\norm{\alpha_w} \ll \delta^{-O(1)}q^K/N$, the function $g_{3,u,v}$ is
    $$
    \ll \delta^{-O(1)}
    (q^K/N) 
    \big(
        {q'}^2q^{2K+2}
            {(
            M/(q'q^K)
            )}^2
    \big)
    =
    \delta^{-O(1)} q^K/M
    $$
    Lipschitz on $[M/(q'q^K)]$. Taking $M' = \floor{c\delta^C M/q^K}$ for appropriate constants $c,C \in \R$, with the use of the triangle inequality we obtain
    \begin{multline*}
    \left|
    \sum_{x \in [N/(q'q^K)]} \sum_{y \in [M/(q'q^K)]} 
    g_{0,u,v}(x) 
    g_{1,u,v}(x+y)
    g_{2,u,v}(x+2y)
        \Bigg(
            \sum_{z \in [M']}
            g_{3,u,v}(y-z)
        \Bigg)
    \right|
    \\
    \gg
    \delta^{O(1)} (N/(q'q^K))(M/(q'q^K))M'.
    \end{multline*}
    Changing the order of summation we obtain
    \begin{multline*}
    \left|
    \sum_{z \in [M']}
    \sum_{x \in [N/(q'q^K)]} 
    \sum_{y \in [M/(q'q^K)]+z} 
    g_{0,u,v}(x) 
    g_{1,u,v}(x+y+z)
    g_{2,u,v}(x+2y+2z)
    g_{3,u,v}(y)
    \right|
    \\
    \gg
    \delta^{O(1)} (N/(q'q^K))(M/(q'q^K))M'.
    \end{multline*}
    Since all functions are $1$-bounded and $z \leq M'$, the inequality 
    \begin{multline*}
    \Bigg|
    \sum_{y \in [M/(q'q^K)]+z} 
    g_{0,u,v}(x) 
    g_{1,u,v}(x+y+z)
    g_{2,u,v}(x+2y+2z)
    g_{3,u,v}(y)
    \ -
    \\
    \sum_{y \in [M/(q'q^K)]} 
    g_{0,u,v}(x) 
    g_{1,u,v}(x+y+z)
    g_{2,u,v}(x+2y+2z)
    g_{3,u,v}(y)
    \Bigg| 
    \ll 2M' \ll \delta^{O(1)}(M/(q'q^K))
    \end{multline*}
    holds, hence
    \begin{multline*}
    \left|
    \sum_{z \in [M']}
    \sum_{x \in [N/(q'q^K)]} 
    \sum_{y \in [M/(q'q^K)]} 
    g_{0,u,v}(x) 
    g_{1,u,v}(x+y+z)
    g_{2,u,v}(x+2y+2z)
    g_{3,u,v}(y)
    \right|
    \\
    \gg
    \delta^{O(1)} (N/(q'q^K))(M/(q'q^K))M'.
    \end{multline*}
    Changing the order of summation and applying the triangle inequality, we obtain
    \begin{multline*}
    \delta^{O(1)} (N/(q'q^K))(M/(q'q^K))M' 
    \\
    \ll
    \sum_{y \in [M/(q'q^K)]}
    \left|
    g_{3,u,v}(y)
    \sum_{z \in [M']}
    \sum_{x \in [N/(q'q^K)]}  
    g_{0,u,v}(x) 
    g_{1,u,v}(x+y+z)
    g_{2,u,v}(x+2y+2z)
    \right|
    \\
    =
    \sum_{y \in [M/(q'q^K)]}
    \left|
    \sum_{z \in [M']}
    \sum_{x \in [N/(q'q^K)]}  
    g_{0,u,v}(x) 
    g_{1,u,v}(x+y+z)
    g_{2,u,v}(x+2y+2z)
    \right|
    .
    \end{multline*}
    By the pigeonhole principle, there exits $y \in [M/(q'q^K)]$ satisfying
    $$
    \delta^{O(1)} (N/(q'q^K))M' 
    \ll
    \left|
    \sum_{x \in [N/(q'q^K)]}
    \sum_{z \in [M']}  
    g_{0,u,v}(x) 
    g_{1,u,v}(x+y+z)
    g_{2,u,v}(x+2y+2z)
    \right|
    .
    $$
    Applying the triangle inequality to split sum over interval $x_0+[N/(q'q^K)]$, where $x_0\in [M']$, onto sums over intervals of length $M'$, we obtain
    \begin{multline*}
    \delta^{O(1)} (N/(q'q^K)){M'}^2 
    \\
    \ll
    \sum_{p \in [N/(q'q^K)]}
    \left|
    \sum_{x \in [M']}
    \sum_{z \in [M']}  
    g_{0,u,v}(p+x) 
    g_{1,u,v}(p+x+y+z)
    g_{2,u,v}(p+x+2y+2z)
    \right|
    .
    \end{multline*}
    By popularity principle there exists a subset $\mathcal{P} \subseteq [N/(q'q^K)]$ satisfying $\abs{\mathcal{P}} \gg \delta^{O(1)}N/q^K$ such that for each $p \in \mathcal{P}$, we have
    $$
     \delta^{O(1)} {M'}^2 
    \ll
    \left|
    \sum_{x \in [M']}
    \sum_{z \in [M']}  
    g_{0,u,v}(p+x) 
    g_{1,u,v}(p+x+y+z)
    g_{2,u,v}(p+x+2y+2z)
    \right|
    .
    $$
    Substituting
    \begin{align*}
        h_0(x) & = g_{0,u,v}(p+x) \cdot 1_{[M']}(x) \\
        h_1(x) & = g_{1,u,v}(p+y+x) \cdot 1_{[2M']}(x) \\
        h_2(x) & = g_{2,u,v}(p+2y+x) \cdot 1_{[3M']}(x)
    \end{align*}
    we obtain
    $$
     \delta^{O(1)} {M'}^2 
    \ll
    \left|
    \sum_{x \in [M']}
    \sum_{z \in [M']}  
    h_{0}(x) 
    h_{1}(x+z)
    h_{2}(x+2z)
    \right|
    .
    $$
    Applying Lemma \ref{LEM: niepełny operator zliczający 3AP daje korelacje z fazą liniową} we obtain
    $
    \delta^{O(1)} M' \ll {\lVert \widehat{h}_i \rVert}_{\infty}
    $
    for $i =0,1,2$. Since all $h_i$-s are supported on $[3M']$, for each $i$ there exists $j_i \in \{0,1,2\}$ and $\alpha_i$ such that
    $$
    \delta^{O(1)} M' \ll \sum_{x \in [j_iM'+1,(j_i+1)M']} h_i(x) e(\alpha_i x).
    $$
    Applying Lemma \ref{LEMMA KORELACJA Z FAZĄ DAJE DUŻĄ NORME U 2}, expanding the definition of $h_i$, summing over all $p \in \mathcal{P}$, and extending by nonnegativity yields
    $$
    \delta^{O(1)} (N/(q'q^K)) {M'}^3
    \ll
    \sum_{p \in [N/(q'q^K)]} 
    {
        \lVert 
            g_{i,u,v}
        \rVert
    }^{4}_{U^2(p+[M'])}.
    $$
    Finally, expanding the definition of $g_i$, summing over all $(u,v) \in \mathcal{U}$ and extending by nonnegativity we deduce
    $$
    \delta^{O(1)} N {M'}^3
    \ll
    \sum_x \norm{f_i}^{4}_{U^2(x+q'q^K\cdot[M'])}.
    $$
\end{proof}

We may restate the above result in a form more suitable for further application. 

\begin{lem} \label{LEM: inv thm for first three function in usable form}
    If $N,q \in \N$ and $\delta > 0 $ satisfy $ {(q/\delta)}^{O(1)} \ll N$, then for all 1-bounded functions $f_{0}, f_{1}, f_{2}, f_{3}\ \colon \ \Z \rightarrow \C$ supported in $[N]$, the bound 
    $$
    \abs{\sum_{x} \sum_{y\in \N} f_{0}(x)f_{1}(x+y)f_{2}(x+2y)f_{3}(x+q^2y^{3})} 
    \geq
    \delta N M
    ,
    $$
    where $M = \lfloor \sqrt[3]{N/q^2} \rfloor$, implies the existence of positive integers $q' \ll \delta^{-O(1)}$, $K = O(1)$ and $M' \gg
    \delta^{O(1)} M/q^K
    $ such that for each $i \in \{0,1,2\}$ there exists a function $\phi_i\ \colon \ \Z \rightarrow \C$ satisfying
    $$
    \sum_x f_i(x) \phi_i(x) \gg \delta^{O(1)} N  
    $$
    and there exists numbers $x_i \in [M']$ and $\alpha_{u,p,i},\beta_{u,p,i} \in \T$, where $u \in [q'q^K]$ and $p \in \Z$, such that for any $x \in [M']$, we have the formula
    $$
    \phi_i\big(u + q'q^K(M'p+x_i+x)\big) 
    = e(\alpha_{u,p,i}x + \beta_{u,p,i})
    .
    $$
\end{lem}
\begin{proof}
    Applying Theorem \ref{THM: inverse theorem for the first three functions stated as the local gowers norm control} we obtain
    $$
    \sum_x \norm{f_i}^{4}_{U^2(x+q'q^K\cdot[M'])} \gg \delta^{O(1)} N {M'}^3
    $$
    for all $i \in \{0,1,2\}$. Rewriting the sum we obtain
    \begin{multline*}
    \delta^{O(1)} N {M'}^3
    \ll
    \sum_{u \in [q'q^K]}
        \sum_x
        \norm{f_i}^{4}_{U^2(u+q'q^Kx+q'q^K\cdot[M'])}
    \\
    =
    \sum_{u \in [q'q^K]}
        \sum_{x_0 \in [M']}
        \sum_{p}
        \norm{f_i}^{4}_{U^2(u+q'q^Kx_0+q'q^KM'p+q'q^K\cdot[M'])}
    \\
    =
    \sum_{x_0 \in [M']}
    \sum_{u \in [q'q^K]}
        \sum_{p }
        \norm{f_i}^{4}_{U^2(u+q'q^Kx_0+q'q^KM'p+q'q^K\cdot[M'])}
        .
    \end{multline*}
    By the pigeonhole principle there exists $x_i \in [M']$ satisfying
    $$
    \delta^{O(1)} (N/M') {M'}^3
    \ll
    \sum_{u \in [q'q^K]}
        \sum_{p }
        \norm{f_i}^{4}_{U^2(u+q'q^Kx_i+q'q^KM'p+q'q^K\cdot[M'])}
    .
    $$
    The popularity principle now shows that there is a subset $\mathcal{UP} \subset [q'q^K] \times \Z$ satisfying \\
    $\abs{\mathcal{UP}} \gg \delta^{O(1)}N/M'$ such that for each $(u,p) \in \mathcal{UP}$ the inequality
    $$
    \delta^{O(1)}{M'}^3
    \ll
    \norm{f_i}^{4}_{U^2(u+q'q^Kx_i+q'q^KM'p+q'q^K\cdot[M'])}
    $$
    holds. Employing Theorem \ref{Inverse theorem for the $U^2$-norm in Z} yields
    $$
    \delta^{O(1)}{M'}
    \ll
    \sum_{x \in [M']} 
    f_i\big(u+q'q^K(M'p+x_i+x)\big) e(\alpha_{u,p,i}x)
    $$
    for appropriate $\alpha_{u,p,i} \in \T$.  For $(u,p) \notin \mathcal{UP}$ choose $\alpha_{u,p,i} = 0$ and $\beta_{u,p,i}$ to be such that
    $$
    \sum_{x \in [M']} 
    f_i\big(u+q'q^K(M'p+x_i+x)\big) e(\alpha_{u,p,i}x + \beta_{u,p,i}) \geq 0
    .
    $$
    Taking $\phi_i$ as in the statement of the lemma, summing over $(u,p) \in \mathcal{UP}$, and extending by nonnegativity, we obtain
     $$
    \delta^{O(1)} N \ll
    \sum_x f_i(x) \phi_i(x),
    $$
    and hence the proof is complete.
\end{proof}

\section{Inverse theorem for $f_3$}

\begin{thm}[Inverse theorem for $f_3$]\label{inverse theorem for function $f_3$} If $N,q \in \N$ and $\delta > 0$ satisfy ${(q/\delta)}^{O(1)} \spc {\ll} \spc N$, then for all 1-bounded functions $f_0, f_1, f_2, f_3 \ \colon \  \bb{Z} \rightarrow \bb{C}$ supported on $[N]$ the bound
	$$
	\Bigg{|} \sum_{x } \sum_{y \in \N} \spc
	f_0({x}) f_1({x} + y) f_2({x} + 2y) f_3({x} + q^2y^3)
	\Bigg{|} \geq \delta N M,
	$$
	where $M=\lfloor \sqrt[3]{N/q^2} \rfloor$, implies the existence of positive integers $q''' \ll \delta^{-O(1)}$, $L = O(1)$ and $N' \gg (\delta/q)^{O(1)} N^{1/3}$ such that
    $$
    \sum_x 
    \Bigg|
    \sum_{y \in [N']} f_3(x+q'''q^Ly)
    \Bigg|
    \gg
    \delta^{O(1)}NN'.
    $$
\end{thm}

\begin{proof}
    We use the Cauchy-Schwarz inequality to replace $f_0$ with the dual function $F_0$ and then employ Lemma \ref{LEM: inv thm for first three function in usable form}. This allows us to replace $f_0$ by the function $\phi_0$ described in the conclusion of Lemma \ref{LEM: inv thm for first three function in usable form} in the counting operator while preserving its' largeness. Repeating the same procedure for $f_1$ and $f_2$ we obtain functions $\phi_i \colon \Z \rightarrow \C$ ($i=0,1,2$), $K =O(1)$ and $x_0,x_1,x_2 \in [M']$, such that for any $u \in [q'q^K]$, $p \in \Z$ and $x \in [M']$, we have
    $$
    \phi_i\big(u + q'q^K(M'p+x_0+x)\big) 
    = e(\alpha_{u,p,i}x + \beta_{u,p,i})
    ,
    $$
    for some $\alpha_{u,p,i},\beta_{u,p,i} \in \T$. As a result of this process, we have
    $$
    \delta^{O(1)} NM
    \ll
    \sum_{x } \sum_{y \in \N} \spc
	\phi_0(x) 1_{[N]}(x) \phi_1({x} + y) \phi_2({x} + 2y) f_3({x} + q^2y^3).
    $$
    
    Let $q' = q_0'q_1'q_2'$. Then all of the functions $\phi_i$ are made of phases on arithmetic progressions of common difference $q'q^K$ and length $M'/q'$ which satisfies $ \gg \delta^{O(1)}M/q^K$. As a consequence 
    \begin{multline*}
    \delta^{O(1)} NM
    \ll
    \sum_{u,v \in [q'q^K]}
    \Bigg|
    \sum_{x } \sum_{y \in \N} \spc
	\phi_0(u+q'q^Kx) 1_{[N]}(u+q'q^Kx) 
    \phi_1(u+q'q^Kx + v+q'q^Ky)
    \\
    \phi_2\big(u+q'q^Kx + 2(v+q'q^Ky)\big) 
    f_3\big(u+q'q^Kx + q^2{(v+q'q^Ky)}^3\big)
    \Bigg|.
    \end{multline*}
    Substituting
    \begin{align*}
        P_{u,v}(y) & = {q'}^2q^{2K+2}y^3 + 3vq'q^{K+2}y^2 + 3v^2q^2y \\
        \phi_{0,u,v}(x) & = \phi_{0}(u+q'q^Kx) \cdot 1_{[N]}(u+q'q^Kx) \\
        \phi_{1,u,v}(x) & = \phi_{1}(u+v+q'q^Kx) \cdot 1_{[N]}(u+v+q'q^Kx) \\
        \phi_{2,u,v}(x) & = \phi_{2}(u+2v+q'q^Kx) \cdot 1_{[N]}(u+2v+q'q^Kx) \\
        f_{3,u,v}(x) & = f_3(u+q^2v^3+q'q^Kx),
    \end{align*}
    we see that there exists a set $\mathcal{U} \subset {[q'q^K]}^2$ of size 
    $ 
    \gg \delta^{O(1)} {(q'q^K)}^2
    $
    such that for any $(u,v) \in \mathcal{U}$ we obtain
    \begin{multline} \label{InvThmf3 eq0}
    \delta^{O(1)} (N/(q'q^K))(M/(q'q^K))
    \\
    \ll
    \Bigg|
    \sum_{x } \sum_{y \in \N}
    \phi_{0,u,v}(x)
    \phi_{1,u,v}(x+y)
    \phi_{2,u,v}(x+2y)
    f_{3,u,v}\big(x+P_{u,v}(y)\big)
    \Bigg|
    .
    \end{multline}
    
    Take $M''$ to be an integer satisfying $M'' \ll \delta^{O(1)} M'/q'$, yet still being of the size 
    $ 
    \gg \delta^{O(1)} M/(q'q^K)
    $. 
    In fact we choose any integer $M''$ satisfying
    $
    \frac{1}{2C^3} \delta^{3C} M'/q' \leq M'' \leq \frac{1}{C^3} \delta^{3C} M'/q'
    $,
    where $C>0$ is such that the expression $\frac{1}{C}\delta^C (N/(q'q^K))(M/(q'q^K))$ bounds the right hand side of  (\ref{InvThmf3 eq0}). We now divide the range of $x$ into intervals of length $M''$ instead of $M'$. This additional refinement will allow us to find many intervals on which functions $\phi_{i,u,v}$-s behave like a single linear phase. If instead we would have considered intervals of length $M'$, functions $\phi_{1,u,v}$ and $\phi_{2,u,v}$ would be made from several different phases on relevant intervals of lengths $2M'$ and $3M'$. Applying triangle inequality we obtain
    \begin{multline} \label{InvThmf3 eqXD}
    \delta^{O(1)} (N/(q'q^K))(M/(q'q^K))
    \ll
    \sum_{m}
    \Bigg|
    \sum_{x \in [M'']} \sum_{y \in [M/(q'q^K)]}
    \phi_{0,u,v}(mM''+x)
    \\
    \phi_{1,u,v}(mM''+x+y)
    \phi_{2,u,v}(mM''+x+2y)
    f_{3,u,v}\big(mM''+x+P_{u,v}(y)\big)
    \Bigg|
    .
    \end{multline}
    We therefore conclude that there exists a subset $\mathcal{M} \subset \Z$ of size $\gg \delta^{O(1)} N/(q'q^KM'')$ such that for each $m \in \mathcal{M}$ the summand on the right hand side of (\ref{InvThmf3 eqXD}) associated to $m$ satisfies $\gg \delta^{O(1)} M''M/(q'q^K)$.
        
    Since $M' / \left(q' M''\right) \gg  \delta^{-O(1)}$, we can assume that for each $m \in \mathcal{M}$ there exists $p \in \Z$ satisfying
    $
    mM''+[M'']  \subset x_0 + pM'+[M']
    $.
    Similarly, since $3M/(q'q^KM'') \ll \delta^{-O(1)}$, we can assume that $\mathcal{M}$ lies in an arithmetic progression $w+\lfloor3M/(q'q^KM'')\rfloor\Z$.
 
    With this preparation we are ready to decide which phases to consider in the functions $\phi_{1,u,v}$ and $\phi_{2,u,v}$. The triangle inequality yields
    \begin{multline} \label{InvThmf3 eqXDD}
    \delta^{O(1)} M''M/(q'q^K)
    \ll
    \sum_{n \in [M/(q'q^KM'')]}
    \Bigg|
    \sum_{x \in [M'']} \sum_{y \in [M'']}
    \phi_{0,u,v}(mM''+x)
    \\
    \phi_{1,u,v}(mM''+x+nM''+y)
    \phi_{2,u,v}(mM''+x+2nM''+2y)
    f_{3,u,v}\big(mM''+x+P_{u,v}(nM''+y)\big)
    \Bigg|
    .
    \end{multline}
    The popularity principle shows that there are at least $\gg \delta^{O(1)} M/(q'q^KM'')$ numbers $n$ such that the summand in (\ref{InvThmf3 eqXDD}) associated to them satisfies $\gg \delta^{O(1)} M''^2$. Since the inequality $M'/\left(q' M''\right)  \gg  \delta^{-O(1)}$ holds, among these $n$-s there is a number $n_m \in [M/(q'q^KM'')]$ for which there are integers $p_1,p_2 \in \Z$ satisfying
    $
    (m+in) M'' + [(i+1)M'']  
    \subset
    x_i + p_i M' + [M']
    $. 
    
    Let $\gamma_m \in \T$ be such that the summand from (\ref{InvThmf3 eqXDD}) associated to $n_m$  multiplied by $e(3 \gamma_m )$ is a nonnegative real. Define
    \begin{align*}
        \varphi_{i,u,v}(x) & = \begin{cases}
            e(\gamma_m)\cdot \phi_{i,u,v}(x) & \text{if } x \in (m+in_m)M'' + [(i+1)M''] \text{ for some } m \in \mathcal{M} \\
            0 & \text{otherwise} 
        \end{cases} 
    \end{align*}
    and observe that 
    $$
    \delta^{O(1)}
    (N/(q'q^K))(M/(q'q^K))
    \\
    \ll
    \sum_{x} \sum_{y \in \N}
    \varphi_{0,u,v}(x)
    \varphi_{1,u,v}(x+y)
    \varphi_{2,u,v}(x+2y)
    f_{3,u,v}\big(x+P_{u,v}(y)\big)
    .
    $$
    The advantage in considering functions $\varphi_{i,u,v}$ instead of $\phi_{i,u,v}$ comes from the fact that once we prove that multiplying functions 
    $\varphi_{0,u,v}(x)
    \varphi_{1,u,v}(x+y)
    \varphi_{2,u,v}(x+2y)$
    yields values very close to constants when $x$ and $y$ vary over appropriate intervals of length $[M'']$, thanks to the choice of $\mathcal{M}$ we deduce the correlation of $f_{3,u,v}$ with this almost constant function. Previously, the functions $\phi_{i,u,v}$ contained also irrelevant phases which did not produce correlation with $f_{3,u,v}$. Those phases should be viewed as the noise which disturbed our initial measurements.
    
    Observe that $P_{u,v}$ has $(3,q'q^{K+1})$-coefficients and the main term ${q'}^2q^{K+1}$. We can therefore employ Theorem \ref{THM: U^1 control of F_3} which yields the existence of $q_{u,v}'' \ll \delta^{-O(1)}$, $K' = O(1)$ and the function $\phi_{3,u,v}$ which is $\ll \delta^{-O(1)} q'q_{u,v}''q^{K+K'}/N$ Lipschitz along arithmetic progressions of the common difference $q_{u,v}''q^{K'}$ and satisfies
    \begin{multline*}
    \delta^{O(1)} (N/(q'q^K))(M/(q'q^K)) \ll
    \sum_{x \in [N/(q'q^K)]} \sum_{y \in [M/(q'q^K)]}
    \varphi_{0,u,v}(x)
    \varphi_{1,u,v}(x+y)
    \varphi_{2,u,v}(x+2y)
    \\
    \phi_{3,u,v}\big(x+P_{u,v}(y)\big)
    .
    \end{multline*}
    
    Before we proceed further, as for any $(u,v) \in \mathcal{U}$ the inequality 
    $q''_{u,v} \ll \delta^{-O(1)}$ holds, there exists $q'' \ll \delta^{-O(1)}$ and a subset $\mathcal{U}' \subset \mathcal{U}$ still having size 
    $ 
    \gg \delta^{O(1)} {(q'q^K)}^2
    $
    and
    such that for each $(u,v) \in \mathcal{U'}$ we have $q''_{u,v} = q''$. From now on we only focus on pairs $(u,v)$ from this refined subset $\mathcal{U'}$. Thanks to the definitions of $\varphi_{i,u,v}$ and the triangle inequality, we obtain
    \begin{multline} \label{InvThmf3 eq1}
    \delta^{O(1)} (N/(q'q^K))(M/(q'q^K)) \ll
    \sum_{m \in \mathcal{M}''}
    \Bigg|
    \sum_{x \in mM''+[M'']} \ \sum_{y \in n_mM''+[M'']}
    \varphi_{0,u,v}(x)
    \varphi_{1,u,v}(x+y)
    \\
    \varphi_{2,u,v}(x+2y)
    \phi_{3,u,v}\big(x+P_{u,v}(y)\big)
    \Bigg|
    .
    \end{multline}

    Employing the popularity principle we obtain a subset $\mathcal{M}' \subset \mathcal{M}$ still of size
    \\
    $
    \gg \delta^{O(1)}N/(q'q^KM'')
    $
    and such that for each $m \in \mathcal{M}'$ the summand from (\ref{InvThmf3 eq1}) associated to $m$ satisfies $\gg \delta^{O(1)} {M''}^2 $.
    
    By the definition of $\varphi_{i,u,v}$-s, there exist $\alpha_{i,m},\beta_{i,m} \in \T$ such that for any $x \in mM''+[M'']$ and $y \in n_mM''+[M'']$ we have
    $
    \varphi_{i,u,v}(x+iy) = e\big(\alpha_{i,m}(x+iy)+ \beta_{i,m}\big)
    $.
    The pigeonhole principle gives us the existence of $y \in n_mM''+[M'']$ satisfying
    \begin{align} \label{Ineq: AAAA}
    \delta^{O(1)} {M''} 
    \ll
    \Bigg|
    \sum_{x \in mM''+[M'']}
    e
    \big(
        (
            \alpha_{0,m} + 
            \alpha_{1,m} +
            \alpha_{2,m} 
        )
        x
    \big)
    \phi_{3,u,v}\big(x+P_{u,v}(y)\big)
    \Bigg|
    .
    \end{align} 
    
    Since $\phi_{3,u,v}$ is $\ll \delta^{-O(1)} q'q''q^{K+K'}/N$ Lipschitz along arithmetic progressions of the common difference $q''q^{K'}$, we approximate it by a constant on the right hand side of (\ref{Ineq: AAAA}), obtaining only $\ll \delta^{O(1)} q'q^{K} {M''}^2 /N$ error in the process. Employing Weyl's inequality yields
    \begin{align} \label{InvThmf3 cancelation}
    \norm{
        q''q^{K'}
        (
            \alpha_{0,m} + 
            \alpha_{1,m} +
            \alpha_{2,m} 
        )
    }
    \ll
    \delta^{-O(1)} q''q^{K'}/M''
    \end{align}
    for every $m \in \mathcal{M}'$. This is the previously described cancellation of phases.
    
    Popularity principle gives us a subset 
    $
    \mathcal{Y}_m \subset n_mM''+[M''] \subset [M/(q'q^K)]
    $
    satisfying 
    $\abs{\mathcal{Y}_m} \gg \delta^{O(1)}M/(q'q^K)$ 
    and such that for every $y \in \mathcal{Y}_m$ we have
    \begin{align} \label{InvThmf3 eq2}
    \delta^{O(1)} {M''}
    \ll
    \Bigg|
    \sum_{x \in mM''+[M'']}
    e
    \big(
        (
            \alpha_{0,m} + 
            \alpha_{1,m} +
            \alpha_{2,m} 
        )
        x
    \big)
    f_{3,u,v}\big(x+P_{u,v}(y)\big)
    \Bigg|
    .
    \end{align}
    Using double counting and the pigeonhole principle we find a number $y_0 \in [M/(q'q^K)]$ and a subset $\mathfrak{M} \subseteq \mathcal{M}'$ still of the size 
    $
    \gg
    \delta^{O(1)}N/(q'q^KM'')
    $ 
    and such that for each $m \in \mathfrak{M}$, (\ref{InvThmf3 eq2}) holds for $y=y_0$.
    
    Choose $N'$ to be an integer satisfying 
    $
    N' \ll \delta^{O(1)} \left(M''/q'q''q^{K+K'}\right)
    $, 
    yet still of the size 
    $ \gg \delta^{O(1)} M/q'q''q^{(K+K')}$.
Via the cancellation of phases (\ref{InvThmf3 cancelation})
we obtain
    $$
    \delta^{O(1)} M''N' 
    \ll
    \sum_{x \in mM''+[M'']}
    \Bigg|
    \sum_{z \in [N']} f_{3,u,v}
        \big(
            x+q''q^{K'}z+P_{u,v}(y_0)
        \big)
    \Bigg|
    .
    $$
    Summing over $m \in \mathfrak{M}$ and extending by nonnegativity we obtain
    $$
    \delta^{O(1)} (N/(q'q^K)) N'
    \ll
    \sum_x
    \Bigg|
    \sum_{z \in [N']} f_{3,u,v}
        \big(
            x+q''q^{K'}z
        \big)
    \Bigg|
    .
    $$
    Now define 
    $\mathcal{V} = \{ u+q^2v^3 \ (\text{mod } q'q^K) \colon (u,v) \in \mathcal{U}' \}$.
    Observe that since the function $(u,v) \mapsto u+q^2v^3$ is $q'q^K$ to $1$, we obtain the bound $\abs{\mathcal{V}} \gg \delta^{O(1)}q'q^K$. Expanding the definition of $f_{3,u,v}$, summing over all elements of $\mathcal{V}$, and extending by nonnegativity we obtain
    $$
    \delta^{O(1)} N N'
    \ll
    \sum_x
    \Bigg|
    \sum_{z \in [N']} f_{3}
        \big(
            x+q'q''q^{K+K'}z
        \big)
    \Bigg|
    .
    $$
    Putting $q''' = q'q''$ and $L=K+K'$ completes the proof. 
\end{proof}

\section{Density Increment via the Theorem of Varnavides}
We now deduce a density increment from Theorem~\ref{inverse theorem for function $f_3$}. The main step is to find many three-term arithmetic progressions whose common differences are small enough that the fourth point $x+q^2y^3$ still lies in $[N]$. A Varnavides argument provides this count and hence the correlation needed to apply the inverse theorem.
Throughout this section we retain the notation
\[
M=\left\lfloor (N/q^2)^{1/3}\right\rfloor.
\]
Let $r_3(n)$ be the largest size of a subset of $[n]$ containing no nontrivial three-term arithmetic progression. Choose a nonincreasing function $s_3 \ \colon \ \{3,4,\ldots\}\rightarrow (0,1]$ such that
\[
\frac{r_3(n)}{n}\leq s_3(n),\qquad \lim_{n\to\infty}s_3(n)=0.
\]
For $0<\rho<1$, put
$
M_3(\rho)=\min\{n\geq 3:s_3(n)\leq\rho\}
$.
Then $r_3(M_3(\rho))\leq\rho M_3(\rho)$.

We begin with the following quantitative form of Varnavides' theorem~\cite{Varnavides1959Certain}. An analogous result, in a slightly different formulation, appears in~\cite[Theorem 2.1]{RW19}.
\begin{lem}\label{lem:varnavides}
Let $0<\rho\leq 1$ and $I\subset\Z$ be an interval of length $L$. Suppose that a subset  $B\subset I$ satisfies $\abs{B}\geq\rho L$. If
$
L\gg M_3(\rho/8)^2
$, then
\[
\abs{\{(x,y): y> 0,\ x,x+y,x+2y\in B\}} \gg  M_3(\rho/8)^{-2} L^2.
\]
\end{lem}

We next apply Lemma \ref{lem:varnavides} on short intervals to obtain
the correlation required by Theorem
\ref{inverse theorem for function $f_3$}. We consider configurations
of the form
\begin{equation}\label{real considered configuration}
x,x+y,x+2y,x+q^2y^3,\qquad y > 0.
\end{equation}
\begin{lem}\label{fourth-slot-reduction}
Let $A\subset[N]$ have density $0<\alpha\leq1$, and put
$f_A=1_A-\alpha1_{[N]}$, $R_\alpha=M_3(\alpha/32)$.
Suppose that $A$ contains no configuration
\eqref{real considered configuration}. Then either
$
N^{1/3}\ll\alpha^{-4/3}R_\alpha^2q^{2/3}
$,
or
\[
\left|\sum_x\sum_{y\in[M]} 1_A(x)1_A(x+y)1_A(x+2y)f_A(x+q^2y^3)\right|
\gg \alpha^{7/3}R_\alpha^{-2}NM.
\]
\end{lem}

\begin{proof}
Assume that $N^{1/3}>C\alpha^{-4/3}R_\alpha^2q^{2/3}$ for a sufficiently large absolute constant $C$, so that $M\asymp N^{1/3}q^{-2/3}$. Write
\[
S=\sum_x\sum_{y\in[M]} 1_A(x)1_A(x+y)1_A(x+2y)1_{[N]}(x+q^2y^3).
\]
Since $A$ contains no configuration \eqref{real considered configuration}, substituting $1_A=\alpha1_{[N]}+f_A$ in the last term gives
\[
0=\alpha S + \sum_x\sum_{y\in[M]} 1_A(x)1_A(x+y)1_A(x+2y)f_A(x+q^2y^3).
\]
It therefore suffices to show that
\[
S\gg\alpha^{4/3}R_\alpha^{-2}NM.
\]

To bound $S$ from below, we count progressions in short intervals away
from the right endpoint of $[N]$. Fix a sufficiently small absolute
constant $c>0$, and set
\[
L=\lfloor c\alpha^{1/3}M\rfloor,\qquad J=[N-q^2L^3].
\]
The lower bound on $N$ gives $L\asymp\alpha^{1/3}M$ and $L\gg R_\alpha^2$. It also gives $L\leq\alpha N/4$, since
\[
\frac{L}{\alpha N} \leq c\alpha^{-2/3}q^{-2/3}N^{-2/3}\leq\frac14.
\]
Moreover,
\[
q^2L^3\leq c^3\alpha q^2M^3\leq\frac{\alpha N}{4}, \qquad \abs{A\cap J}\geq\frac{3\alpha N}{4}.
\]
Partition $J$ into intervals of length $L$ and discard the remaining interval of length less than $L$. Call an interval $I$ in this partition good if $\abs{A\cap I}\geq\alpha L/4$, and let $G$ be the total length of the good intervals. Then
\[
\frac{3\alpha N}{4}\leq\abs{A\cap J} \leq \frac{\alpha N}{4}+G+L,
\]
so $G\geq\alpha N/4$.

For each good interval $I$, Lemma \ref{lem:varnavides}, applied with $\rho=\alpha/4$, gives
\[
\abs{\{(x,y)\in\Z\times\N:x,x+y,x+2y\in A\cap I\}} \gg R_\alpha^{-2}L^2.
\]
Each of these progressions has $y\leq L\leq M$ and $x\in J$, whence
\[
1\leq x+q^2y^3\leq x+q^2L^3\leq N.
\]
It therefore contributes to $S$. Summing over the good intervals, we obtain
\[
S\gg\frac{G}{L}R_\alpha^{-2}L^2 \gg\alpha R_\alpha^{-2}NL \gg\alpha^{4/3}R_\alpha^{-2}NM,
\]
as required.
\end{proof}

\begin{rem}\label{restating lemma 8.2}
We use the following consequence of Raghavan's bound \cite{Rag26}:
there are absolute constants $C_0,c_0,a>0$ such that
\[
\frac{r_3(n)}n\leq C_0\exp(-c_0(\log n)^a).
\]
Taking $s_3(n)=\min\{1,C_0\exp(-c_0(\log n)^a)\}$, we obtain
\[
R_\alpha\ll\exp\big(C_1(\log(1/\alpha))^{1/a}\big)
\]
for an absolute constant $C_1$. Lemma \ref{fourth-slot-reduction} therefore
gives an absolute constant $C\geq1$ such that either
$
N\ll\exp\big(C(\log(1/\alpha))^C\big)q^2
$, or
\[
\left|\sum_x\sum_{y\in[M]} 1_A(x)1_A(x+y)1_A(x+2y)f_A(x+q^2y^3)\right|
\gg\exp\big(-C(\log(1/\alpha))^C\big)NM.
\]
\end{rem}

Combining this estimate with Theorem~\ref{inverse theorem for function $f_3$}, we obtain the density increment.

\begin{lem}[Density increment]\label{density increment}

There is an absolute constant $C\geq1$ with the following property. Let $N \in \N$, $A\subset[N]$ have density $0<\alpha\leq 1$ and assume that $A$ does not contain any configuration \eqref{real considered configuration}. Put
\[
E_\alpha=C\exp\big(C(\log(1/\alpha))^C\big).
\]
Then either $N\leq E_\alpha q^C$, or there exist positive integers $q''' \leq E_{\alpha}$, $L\leq C$, $N'\geq E_\alpha^{-1}q^{-C}N^{1/3}$ and $x\in\Z$ such that
\[
\frac{\abs{A\cap (x+q'''q^L\cdot[N']) }}{N'}
\geq
\alpha+E_\alpha^{-1}.
\]
\end{lem}
\begin{proof}
Let $C_0$ be a sufficiently large constant and put
\[
\delta=C_0^{-1}\exp\big(-C_0(\log(1/\alpha))^{C_0}\big).
\]
Suppose that $N>E_\alpha q^C$. For $C$ sufficiently large, Remark~\ref{restating lemma 8.2} gives
\[
\left|\sum_x\sum_{y\in[M]} 1_A(x)1_A(x+y)1_A(x+2y)f_A(x+q^2y^3)\right|\geq\delta NM.
\]
The functions $1_A$ and $f_A$ are 1-bounded and supported on $[N]$, and the summand vanishes for $y>M$. The lower bound on $N$ also ensures the size condition in Theorem \ref{inverse theorem for function $f_3$}. Applying that theorem, we obtain positive integers $q'''$, $H$ and a nonnegative integer $L=O(1)$ such that, with $d=q'''q^L$,
\[
q'''\ll\delta^{-O(1)},\qquad H\gg(\delta/q)^{O(1)}N^{1/3}.
\]
We also have
\begin{equation}\label{sum of abs values of weighted characteristic function}
\sum_x\left|\sum_{y\in[H]}f_A(x+dy)\right|\geq\eta NH
\end{equation}
for some $0<\eta\leq1/2$ with $\eta\gg\delta^{O(1)}$.

We use \eqref{sum of abs values of weighted characteristic function} and the identity $\sum_zf_A(z)=0$ to find a shorter progression on which $f_A$ has positive average. Set
\[
K=\left\lfloor\min\left\{\frac{\eta H}{8}, \frac{\eta N}{8d}\right\}\right\rfloor.
\]
Since $d\ll\delta^{-O(1)}q^{O(1)}$, both terms in the minimum are $\gg(\delta/q)^{O(1)}N^{1/3}$. For $C$ sufficiently large, we thus obtain
\[
K\gg(\delta/q)^{O(1)}N^{1/3},\qquad K\leq\eta H/8,\qquad dK\leq\eta N/8.
\]
Partition the terms of each residue class modulo $d$ in $[N]$ into consecutive blocks of length $K$, with a shorter block at the end if necessary. Denote the resulting partition by $\mathcal P$, and write
\[
F(P)=\sum_{z\in P}f_A(z)\qquad(P\in\mathcal P).
\]

For each $x \in \Z$, all but at most two of the blocks meeting $x+d\cdot [H]$ are contained in that progression. The triangle inequality and the bound $|f_A|\leq1$ give
\[
\sum_x\left|\sum_{y\in[H]}f_A(x+dy)\right| \leq H\sum_{P\in\mathcal P}|F(P)|+2NK.
\]
Indeed, a block is contained in at most $H$ of these progressions. For each $z\in[N]$, there are at most $2K$ values of $x$ for which $z\in x+d\cdot[H]$ but its block is not contained in $x+d\cdot[H]$, since one of the two endpoints must then lie in that block. It follows from \eqref{sum of abs values of weighted characteristic function} that
\[
\sum_{P\in\mathcal P}|F(P)|\geq\eta N-\frac{2NK}{H} \geq\frac{3\eta N}{4}.
\]
On the other hand, $f_A$ has mean zero and the blocks partition $[N]$, so
\begin{equation}\label{sum of weighted characteristic function}
\sum_{P\in\mathcal P}F(P)=\sum_zf_A(z)=0,
\end{equation}
hence
\[
\sum_{\substack{P\in\mathcal P:\,F(P)>0}}F(P) =\frac12\sum_{P\in\mathcal P} |F(P)|\geq\frac{3\eta N}{8}.
\]
There are at most $d$ blocks of length less than $K$, and their total contribution to this sum is at most $dK\leq\eta N/8$. Therefore
\[
\sum_{\substack{P\in\mathcal P:\,\abs{P}=K}}
\max\{F(P),0\}\geq\frac{\eta N}{4}.
\]
The blocks of length $K$ have total size at most $N$, so one of them satisfies $F(P)\geq\eta K/4$. For this block,
\[
\frac{\abs{A\cap P}}{K}=\alpha+\frac{F(P)}K \geq\alpha+\frac{\eta}{4}.
\]
Since $P=x+d\cdot [K]\subset[N]$ for some $x\in\Z$, the bounds for $q'''$, $K$ and $\eta$ give the conclusion with $N'=K$, provided $C$ is sufficiently large.
\end{proof}

\section{Proof of the main theorem from Density Increment}

We finish this work with a proof of the Theorem \ref{main result}.
\begin{proof}
    Note that if $A$ lacks configuration
    $$ x,x+y,x+2y,x+q^2y^3, \quad \text{for some} \ y\neq 0$$
    then the set 
    $$ \{x: a+q'''q^Lx \in A\}$$
    lacks configurations of the form   
    $$
    x,x+y,x+2y,x+q^{2(L+1)}q'''^2 y^3.
    $$

    Let $A\subset [N]$ have size $\alpha N$ and lack configuration
    $$ x,x+y,x+2y,x+y^3 \quad \text{for some} \ y\neq 0.$$
    Using the density increment Lemma \ref{density increment}, we obtain a sequence of tuples $(A_i, N_i, q_i)$ for \\
    $i\in [0,m]$, for some $m\in \N$, such that:
    \begin{enumerate}
    \setcounter{enumi}{0}
    \item $(A_0, N_0, q_0)=(A, N, 1)$; 
        \item if $i\geq 1$, then $A_i$ lacks configuration of the form
        $$ x,x+y,x+2y,x+q_0^{2\cdot (L+1)^{i}}q_1^{2\cdot (L+1)^{i-1}}\ldots q_i^2 y^3;$$
    \item $q_i \leq E_{\alpha}$;
    \item $A_i \subset [N_i]$ and for $i\geq 1$ we have
    $$ \frac{\#A_i}{N_i} \geq \alpha+iE_{\alpha}^{-1};$$
    \item for $i\geq 1$ we have the lower bound
    $$ N_i \geq E_{\alpha}^{-1}(q_0^{2\cdot (L+1)^{i-1}}q_1^{2\cdot (L+1)^{i-2}}\ldots q_{i-1}^2)^{-O(1)}N_{i-1}^{1/3}.
    $$
    \end{enumerate}
    By the Lemma \ref{density increment}, we either have 
 \begin{equation}\label{estimate on N_m}
    N_m \leq E_{\alpha}(q_0^{2\cdot (L+1)^{m}}q_1^{2\cdot (L+1)^{m-1}}\ldots q_{m}^2)^{O(1)} 
 \end{equation}
    and the process terminates, or we obtain $(A_{m+1}, N_{m+1}, q_{m+1})$ satisfying properties $(2)-(5)$. 

    The density of $A_m$ inside $N_m$ cannot exceed 1, so the process must terminate at some 
    \begin{equation}\label{estimate on m}
       m\ll E_{\alpha}.
    \end{equation}
    At the stopping point, an assumption \eqref{estimate on N_m} must be satisfied, so we have
   \begin{equation}\label{N_m upper bound}
       N_m \leq E_{\alpha}(q_0^{2\cdot (L+1)^{m}}q_1^{2\cdot (L+1)^{m-1}}\ldots q_{m}^2)^{O(1)} 
   \end{equation}    
The condition $(5)$ of this lemma implies 
\begin{align*}\label{estimate of decrease of N_i}
N_i 
\geq
E_{\alpha}^{-1}(q_0^{2\cdot (L+1)^{m-1}}q_1^{2\cdot (L+1)^{m-2}}\ldots q_{m-1}^2)^{-O(1)}N_{i-1}^{1/3}.
\end{align*} 
  Applying previous inequality iteratively, we obtain
   \begin{align*}
       N_m
 &      \geq N^{1/3^m}\left( E_{\alpha}^{-1}(q_0^{2\cdot (L+1)^{m-1}}q_1^{2\cdot (L+1)^{m-2}}\ldots q_{m-1})^{-O(1)}\right)^{(1+\frac{1}{3}+\frac{1}{9}+\ldots +3^{1-m})} \\
&       \geq N^{1/3^m}E_{\alpha}^{-3/2}(q_0^{2\cdot (L+1)^{m-1}}q_1^{2\cdot (L+1)^{m-2}}\ldots q_{m-1})^{-O(1)}.
   \end{align*}
   Substituting \eqref{N_m upper bound} into the inequality above, we obtain
   \begin{equation*}
     N^{1/3^m}E_{\alpha}^{-3/2}(q_0^{2\cdot (L+1)^{m-1}}q_1^{2\cdot (L+1)^{m-2}}\ldots q_{m-1})^{-O(1)} \leq E_{\alpha} (q_0^{2\cdot (L+1)^{m}}q_1^{2\cdot (L+1)^{m-1}}\ldots q_{m}^2)^{O(1)},  
   \end{equation*}
   which implies the inequality
  \begin{equation*}
     N^{1/3^m} \leq E_{\alpha}^{5/2} (q_0^{2\cdot (L+1)^{m}}q_1^{2\cdot (L+1)^{m-1}}\ldots q_{m}^2)^{O(1)}.  
   \end{equation*}
   Taking logarithms, using an estimate $q_i \leq E_{\alpha} \ll \exp\left( (\log(1/\alpha))^{O(1)} \right)$, substituting \eqref{estimate on m} and using an estimate $\log(1/\alpha) \ll \exp\left( (\log(1/\alpha))^{O(1)} \right)$, we obtain the claimed bound on N.
\end{proof}

\addtolength{\leftmargin}{0cm}
\setlength{\itemindent}{0cm}
\bibliographystyle{plain}
\bibliography{main.bib}

\end{document}